\documentclass[11pt,oneside,reqno]{article}
\usepackage[margin=1.3in]{geometry}

\usepackage{amsmath,amsthm,amssymb,color}
\usepackage[numbers]{natbib}
\usepackage{booktabs}

\RequirePackage{amsthm,amsmath,amsfonts,amssymb,mathrsfs,dsfont,mathtools,thmtools}
\RequirePackage{natbib}
\RequirePackage[colorlinks,citecolor=blue,urlcolor=blue]{hyperref}
\usepackage[capitalize]{cleveref}
\RequirePackage{graphicx}

\crefname{equation}{}{}
\crefformat{equation}{\textup{(#2#1#3)}}
\crefrangeformat{equation}{\textup{(#3#1#4)--(#5#2#6)}}
\crefdefaultlabelformat{#2\textup{#1}#3}
\crefname{lemma}{Lemma}{Lemmas}
\crefname{page}{p.}{pp.}

\usepackage{bbm}
\usepackage{mathrsfs}
\usepackage{graphicx}
\usepackage{tikz}
\usepackage{enumitem}
\usetikzlibrary{arrows, automata}
\usetikzlibrary{calc}
\usetikzlibrary{positioning}

\numberwithin{equation}{section}
\allowdisplaybreaks[4]

\theoremstyle{plain}
\newtheorem{theorem}{Theorem}[section]
\newtheorem{proposition}{Proposition}[section]
\newtheorem{lemma}{Lemma}[section]

\theoremstyle{definition}
\newtheorem{definition}{Definition}[section]

\newtheorem{remark}{Remark}[section]

\makeatletter

\newcount\minute
\newcount\hour
\newcount\hourMins
\def\now{%
\minute=\time%
\hour=\time \divide \hour by 60%
\hourMins=\hour \multiply\hourMins by 60%
\advance\minute by -\hourMins%
\zeroPadTwo{\the\hour}:\zeroPadTwo{\the\minute}%
}
\def\zeroPadTwo#1{\ifnum #1<10 0\fi#1}

\renewcommand{\cite}{\citet}

\def\^#1{\ifmmode {\mathaccent"705E #1} \else {\accent94 #1} \fi}
\def\~#1{\ifmmode {\mathaccent"707E #1} \else {\accent"7E #1} \fi}

\def\*#1{#1^\ast}
\edef\-#1{\noexpand\ifmmode {\noexpand\bar{#1}} \noexpand\else \-#1\noexpand\fi}
\def\>#1{\vec{#1}}
\def\.#1{\dot{#1}}

\def\atop{\@@atop}
\def\*#1{\mathscr{#1}}

\renewcommand{\leq}{\leqslant}
\renewcommand{\geq}{\geqslant}
\newcommand{\eps}{\varepsilon}
\renewcommand{\eps}{\varepsilon}

\newcommand{\diag}{{\mathop{\mathrm{diag}}}}

\newcommand{\IE}{\mathbbm{E}}

\newcommand{\Var}{\mathop{\mathrm{Var}}\nolimits}
\newcommand{\Cov}{\mathop{\mathrm{Cov}}}

\newcommand{\tr}{\mathop{\mathrm{tr}}}

\def\be#1{\begin{equation*}#1\end{equation*}}
\def\ben#1{\begin{equation}#1\end{equation}}

\def\besn#1{\begin{equation}\begin{split}#1\end{split}\end{equation}}

\def\ba#1{\begin{align*}#1\end{align*}}

\def\norm#1{\Vert#1\Vert}

\def\lnorm#1{\left\Vert#1\right\Vert}

\def\mid{\vert}

\def\beqn#1\eeqn{\begin{align}#1\end{align}}
\def\beq#1\eeq{\begin{align*}#1\end{align*}}

\usepackage{graphicx}
\usepackage{latexsym}
\usepackage{amsmath,amsthm,amssymb,amscd}
\usepackage{epsf,amsmath}

\def\E{{\IE}}

\newcommand{\ul}[1]{\underline{#1}}
\newcommand{\ol}[1]{\overline{#1}}
\newcommand{\mcl}[1]{\mathcal{#1}}
\newcommand{\bs}[1]{\boldsymbol{#1}}

\DeclareMathOperator{\Hess}{Hess}
\DeclareMathOperator{\domain}{Dom}
\DeclareMathOperator{\LC}{LC}

\renewcommand\section{\@startsection {section}{1}{\z@}%
{-3.5ex \@plus -1ex \@minus -.2ex}%
{1.3ex \@plus.2ex}%
{\center\small\sc\mathversion{bold}}}

\def\subsection#1{\@startsection {subsection}{2}{0pt}%
{-3.5ex \@plus -1ex \@minus -.2ex}%
{1ex \@plus.2ex}%
{\bf\mathversion{bold}}{#1}}

\def\subsubsection#1{\@startsection{subsubsection}{3}{0pt}%
{\medskipamount}%
{-10pt}%
{\normalsize\itshape}{\kern-2.2ex. #1.}}

\def\blfootnote{\xdef\@thefnmark{}\@footnotetext}

\makeatother

\begin{document}

\title{
Sharp High-dimensional Central Limit Theorems for Log-concave Distributions
}
\author{Xiao Fang and Yuta Koike}
\date{\it The Chinese University of Hong Kong and The University of Tokyo} 
\maketitle

\noindent{\bf Abstract:} Let $X_1,\dots,X_n$ be i.i.d.~log-concave random vectors in $\mathbb R^d$ with mean 0 and covariance matrix $\Sigma$. 
We study the problem of quantifying the normal approximation error for $W=n^{-1/2}\sum_{i=1}^nX_i$ with explicit dependence on the dimension $d$. 
Specifically, without any restriction on $\Sigma$, we show that the approximation error over rectangles in $\mathbb R^d$ is bounded by $C(\log^{13}(dn)/n)^{1/2}$ for some universal constant $C$. Moreover, if the Kannan--Lov\'asz--Simonovits (KLS) spectral gap conjecture is true, this bound can be improved to $C(\log^{3}(dn)/n)^{1/2}$. This improved bound is optimal in terms of both $n$ and $d$ in the regime $\log n=O(\log d)$. 
We also give $p$-Wasserstein bounds with all $p\geq2$ and a Cram\'er type moderate deviation result for this normal approximation error, and they are all optimal under the KLS conjecture. 
To prove these bounds, we develop a new Gaussian coupling inequality that gives almost dimension-free bounds for projected versions of $p$-Wasserstein distance for every $p\geq2$. We prove this coupling inequality by combining Stein's method and Eldan's stochastic localization procedure.  

\medskip



\section{Introduction}

Let $X_1,\dots,X_n$ be i.i.d.~random vectors in $\mathbb R^d$ with mean 0 and covariance matrix $\Sigma=(\Sigma_{jk})_{1\leq j,k\leq d}$. Set $W=n^{-1/2}\sum_{i=1}^nX_i$. The classical central limit theorem (CLT) states that $W$ converges in law to $N(0,\Sigma)$ as $n\to\infty$. 
This paper aims to quantify the convergence rate of this normal approximation with explicit dependence on the dimension $d$. 
It is known that this dependence is crucially determined by how to measure the distance between the law of $W$ and $N(0,\Sigma)$. In this paper, we primarily focus on the uniform distance over rectangles in $\mathbb R^d$. That is,
\[
\rho(W,Z)=\sup_{A\in\mcl R}|P(W\in A)-P(Z\in A)|,
\]
where $Z\sim N(0,\Sigma)$ and $\mathcal{R}:=\{\prod_{j=1}^d [a_j, b_j]: -\infty< a_j< b_j<\infty\}$ is the set of rectangles in $\mathbb R^d$. 
The recent seminal work of \cite{CCK13,CCK17} has shown that, under mild regularity assumptions, one can get a non-trivial bound for $\rho(W,Z)$ even when the dimension $d$ is much larger than the sample size $n$. 
When we allow $\Sigma$ to be degenerate, the currently best known general bound for $\rho(W,Z)$ is as follows: Suppose $\ul\sigma^2=\min_{1\leq j\leq d}\Sigma_{jj}>0$. Suppose also that there exists a constant $B>0$ such that $\E\exp(|X_{1j}|/B)\leq2$ and $\E X_{1j}^4\leq B^2$ for all $j=1,\dots,d$, where $X_{1j}$ is the $j$-th component of $X_1$. Then, according to Theorem 2.1 in \cite{CCKK22}, we have
\ben{\label{eq:cckk}
\rho(W,Z)\leq c\left(\frac{B^2\log^5(dn)}{n}\right)^{1/4},
}
where $c$ is a constant depending only on $\ul\sigma^2$. The bound \eqref{eq:cckk} gives a meaningful estimate for $\rho(W,Z)$ even when $d$ is exponentially larger than $n$, but the dependence on $n$ does not match the classical Berry--Esseen rate $1/\sqrt n$. 
Recently, by exploiting the regularity of $Z$, several authors have succeeded in getting bounds with $1/\sqrt n$ rates up to $\log n$ factors when $\Sigma$ is non-degenerate; see \cite{FaKo21,Lo22,KuRi20,CCK21}. In particular, by Corollary 1.1 in \cite{FaKo21}, if $X_1$ is log-concave (cf.~\cref{def:lc}), then
\ben{\label{eq:fk}
\rho(W,Z)\leq \frac{C}{\sigma_*^2}\sqrt{\frac{\log^3d}{n}}\log n,
}
where $C$ is a positive universal constant and $\sigma_*^2$ is the smallest eigenvalue of the correlation matrix of $W$. 
The bound \eqref{eq:fk} is rate-optimal up to the $\log n$ factor because $\sqrt{\frac{n}{\log^3d}}\rho(W,Z)$ does not vanish as $n\to\infty$ under appropriate growth conditions on $n$ and $d$ when the coordinates of $X_1$ are i.i.d.~and follow a standardized exponential distribution; see Proposition 1.1 in \cite{FaKo21}. 
Corollary 2.1 in \cite{CCK21} gives a similar bound to \eqref{eq:fk} without log-concavity when $X_{1j}$ are uniformly bounded.  

In this paper, we show that a bound of the form $C\sqrt{\log^a(dn)/n}$ for some constants $C,a>0$ is achievable even when $\Sigma$ is degenerate, provided that $X_1$ is log-concave. 
Remarkably, $C$ and $a$ can be taken universally and thus independently of $\Sigma$. 
In addition, if the \textit{Kannan--Lov\'asz--Simonovits (KLS) conjecture} is true, our bound is optimal 
in both $n$ and $d$ in the regime $\log n=O(\log d)$. 
To state the result formally, we introduce some definitions and notations. 

\begin{definition}[Log-concavity]\label{def:lc}
A probability measure $\mu$ on $\mathbb R^d$ is \textit{log-concave} if 
\[
\mu(\theta A+(1-\theta)B)\geq\mu(A)^\theta\mu(B)^{1-\theta}
\]
for any non-empty compact sets $A$ and $B$ of $\mathbb R^d$ and any $\theta\in(0,1)$. 
We say that a random vector $X$ in $\mathbb R^d$ is log-concave if its law $\mcl L(X)$ is log-concave. 
\end{definition}

\begin{definition}[Poincar\'e constant]
A probability measure $\mu$ on $\mathbb R^d$ is said to satisfy a \textit{Poincar\'e inequality} if there exists a constant $\varpi\geq0$ such that
\ben{\label{eq:poincare}
\Var_\mu(h):=\int h^2d\mu-\left(\int hd\mu\right)^2\leq\varpi\int|\nabla h|^2d\mu
}
for every locally Lipschitz function $h:\mathbb R^d\to\mathbb R$ with $h\in L^2(\mu)$. Here,
\[
|\nabla h(x)|:=\limsup_{y\to x}\frac{|h(y)-h(x)|}{|y-x|},\qquad x\in\mathbb R^d.
\]
The smallest constant $\varpi$ satisfying \eqref{eq:poincare} is called the \textit{Poincar\'e constant} of $\mu$ and denoted by $\varpi(\mu)$. By convention, we set $\varpi(\mu):=\infty$ if $\mu$ does not satisfy any Poincar\'e inequality. 
For a random vector $X$ in $\mathbb R^d$, we write $\varpi(X)=\varpi(\mcl L(X))$. 
\end{definition}
%
We denote by $\LC_d$ the set of isotropic (i.e., with zero mean and identity covariance) log-concave probability measures on $\mathbb R^d$. Define
\[
\varpi_d:=\sup_{\mu\in\LC_d}\varpi(\mu).
\]
The KLS conjecture suggests that $\varpi_d$ would be bounded by a universal constant. The currently best known bound is the following one due to \cite{KlLe22} (cf.~Theorem 1.1 and Eq.(7) ibidem):
\ben{\label{kl-est}
\varpi_d\leq C(1\vee\log ^{10}d).
}
Here and below, we use $C$ to denote positive universal constants, which may differ in different expressions. 
We refer to \cite{AGB15} for more background of the KLS conjecture.  

With these notations, our first main result is stated as follows:
\begin{theorem}\label{thm:lc-kol}
Let $n\geq2$ be an integer. 
Let $X_1,\dots,X_n$ be i.i.d.~log-concave random vectors in $\mathbb R^d$ with mean 0 and covariance matrix $\Sigma$, and set $W=n^{-1/2}\sum_{i=1}^nX_i$. 
Let $Z\sim N(0,\Sigma)$. Suppose that $\Sigma$ has rank $r\geq1$. Then 
\ben{\label{eq:lc-kol}
\sup_{A\in\mcl R}|P(W\in A)-P(Z\in A)|\leq C\sqrt{\varpi_r\frac{\log^2(dn)\log(2d)}{n}}.
}
\end{theorem}
The most remarkable feature of the bound \cref{eq:lc-kol} is that the right hand side is bounded by a quantity independent of $\Sigma$ because $\varpi_r\leq\varpi_d$. Note that the log-concavity itself does not impose any restriction on $\Sigma$ because it is invariant under affine transformation. In particular, the bound \eqref{eq:lc-kol} holds even when $\Sigma$ is degenerate. Combining \eqref{eq:lc-kol} with the estimate \eqref{kl-est}, we obtain a bound for $\rho(W,Z)$ with the rate $O(\sqrt{\log^{13}(dn)/n})$. In terms of the dependence on $n$, this improves the bound derived from \eqref{eq:cckk}. 
Moreover, if the KLS conjecture is true, \cref{thm:lc-kol} gives a bound for $\rho(W,Z)$ of the form $C\sqrt{\log^3(dn)/n}$. As shown by Proposition 1.1 in \cite{FaKo21}, this bound is rate-optimal 
in both $n$ and $d$ when $\log n=O(\log d)$. 


To prove \cref{thm:lc-kol}, we construct a coupling of $W$ and $Z$ such that $\|u\cdot(W-Z)\|_p$ enjoys an almost dimension-free bound for any $u\in\mathbb R^d$ and $p\geq1$, where $\cdot$ is the Euclidean inner product and $\|\cdot\|_p$ is the $L^p$-norm with respect to the underlying probability measure. See \cref{sec:prowass} for the precise result.  
Such a bound can be used to control the tail probability of $\max_{1\leq j\leq d}|W_j-Z_j|/\sigma_{j}$ with $\sigma_j:=\sqrt{\Sigma_{jj}}$. As illustrated by Lemma 2.1 in \cite{CCK16}, we can derive a bound for the Kolmogorov distance between $\max_{1\leq j\leq d}(W_j-x_j)/\sigma_j$ and $\max_{1\leq j\leq d}(Z_j-x_j)/\sigma_j$ for any $x\in\mathbb R^d$ from such a control along with an anti-concentration inequality for $\max_{1\leq j\leq d}(Z_j-x_j)/\sigma_j$, and this leads to a bound for $\rho(W,Z)$. Hence, our main technical contribution is derivation of the afore-mentioned coupling inequality for $W$ and $Z$. This new coupling inequality is shown by combining Stein's method and Eldan's stochastic localization procedure as detailed in \cref{sec:proof}.  

Our coupling inequality naturally leads to a bound for the $p$-Wasserstein distance between $W$ and $Z$ for any $p\geq1$, which will be of independent interest. Let us recall the definition of the $p$-Wasserstein distance:
\begin{definition}[$p$-Wasserstein distance]
Let $\mu$ and $\nu$ be two probability measures on $\mathbb{R}^d$. For $p\geq1$, the $p$-Wasserstein distance between $\mu$ and $\nu$ is defined as
\be{
\mcl{W}_p(\mu, \nu)=\left(\inf_\pi\int_{\mathbb{R}^d\times \mathbb{R}^d} |x-y|^p \pi(dx, dy) \right)^{1/p},
}
where $|\cdot|$ denotes the Euclidean norm and $\pi$ is a measure on $\mathbb{R}^d\times \mathbb{R}^d$ with marginals $\mu$ and $\nu$. 
For two random vectors $X$ and $Y$ in $\mathbb{R}^d$, we write $\mcl{W}_p(X, Y)=\mcl{W}_p(\mcl{L}(X), \mcl{L}(Y))$. 
\end{definition}
\begin{theorem}\label{thm:lc-wass}
Under the same assumptions as \cref{thm:lc-kol}, we have 
\ben{\label{lc:wass}
\mcl W_p(W,Z)\leq C\sqrt{\tr(\Sigma)}\left(\sqrt{\varpi_r}\frac{p}{\sqrt n}
+\varpi_r\log(2r)\frac{p^{3/2}}{n}
+\frac{1}{\sqrt{\varpi_r\log (2r)}}\frac{p^{5/2}}{n}\right)
}
for any $p\geq1$. 
Moreover, if $\varpi_r\log^2(2r)\leq cn$ for some positive constant $c$, there exists a constant $C'$ depending only on $c$ such that
\ben{\label{lc:wass-simple}
\mcl W_p(W,Z)\leq C'\frac{p^{2}\sqrt{\tr(\Sigma)\varpi_r}}{\sqrt n}
}
for any $p\geq1$. 
\end{theorem}

In view of \eqref{kl-est}, the condition $\varpi_r\log^2(2r)\leq cn$ will be a rather mild restriction. 
When $X_1$ has independent coordinates and $X_{1j}$ has a non-zero skewness $\gamma_j$ for all $j=1,\dots,d$, $\mcl W_p(W,Z)$ is lower bounded by $c\sqrt{\tr(\Sigma)/n}$ with some positive constant $c$ depending only on $\min_{1\leq j\leq d}|\gamma_j|$ in view of Theorem 1.1 in \cite{Ri11}. Hence the bound \eqref{lc:wass-simple} has optimal dependence on $n,d$ and $\Sigma$ if the KLS conjecture is true. 
Also, when $\Sigma=I_d$, \eqref{kl-est} and \eqref{lc:wass-simple} give an upper bound for $\mcl W_p(W,Z)$ of the form $Cp^2\sqrt{d\log^{10}(2d)/n}$, which improves the currently best known bound $Cp^4d/\sqrt n$ given by Theorem 3.3 in \cite{Fa19}. 
We remark that Theorem 4.1 in \cite{CFP19} implies that the bound \eqref{lc:wass-simple} for $p=2$ and $\Sigma=I_d$ holds without the condition $\varpi_r\log^2(2r)\leq cn$. 
Indeed, by a simplified proof of \cref{thm:lc-wass} for the case $p=2$ using \cref{key0}, we can show $\mcl W_2(W,Z)\leq C'\sqrt{\tr(\Sigma)\varpi_r}/\sqrt n$ without the condition $\varpi_r\log^2(2r)\leq cn$. This bound is completely dimension-free if the KLS conjecture is correct.


Yet another application of our coupling inequality gives the following Cram\'er type moderate deviation result for $\max_{1\leq j\leq d}W_j$:
\begin{theorem}\label{thm:lc-md}
Under the same assumptions as \cref{thm:lc-kol}, suppose additionally that $\sigma_j>0$ for all $j=1,\dots,d$. 
Set 
\ben{\label{def:sigma}
\ol\sigma=\max_{1\leq j\leq d}\sigma_j,\qquad\ul\sigma=\min_{1\leq j\leq d}\sigma_j.
}
Then, there exist universal constants $c\in(0,1)$ and $C>0$ such that, for
\[
\frac{\ol\sigma^2\varpi_r\log^3(3d)}{\ul\sigma^2n}\leq c,\qquad
0\leq x\leq\ul\sigma\left(\frac{\ul\sigma^2n}{\ol\sigma^2\varpi_r}\right)^{1/6},
\]
we have
\ben{\label{eq:lc-md}
\left|\frac{P(\max_{1\leq j\leq d}W_j>x)}{P(\max_{1\leq j\leq d}Z_j>x)}-1\right|\leq C\left(1+\frac{x}{\ul\sigma}\right)\left(\log\left(dn\right)+\frac{x^2}{\ol\sigma^2}\right)\frac{\ol\sigma}{\ul\sigma}\sqrt{\frac{\varpi_r}{n}}.
}
\end{theorem}
Note that, applying the result to $(W^\top,-W^\top)^\top$, we can replace $\max_{1\leq j\leq d}W_j$ and $\max_{1\leq j\leq d}Z_j$ in \eqref{eq:lc-md} with $\max_{1\leq j\leq d}|W_j|$ and $\max_{1\leq j\leq d}|Z_j|$, respectively. 
Corollary 5.1 in \cite{KMB21} gives a Cram\'er type moderate deviation result for $\max_{1\leq j\leq d}|W_j|$ with the bound of the form $K\{(1+x)^6\log^{16}(3d)/n\}^{1/6}$ when coordinates of $X_1$ are sub-exponential, where $K$ is a positive constant depending only on $\ol\sigma$, $\ul\sigma$ and sub-exponential norms of $X_{1j}$. 
In the meantime, denoting by $K'$ a positive constant depending only on $\ol\sigma$ and $\ul\sigma$, we can bound the right hand side of \eqref{eq:lc-md} as
\ba{
K'(1+x)(\log(dn)+x^2)\sqrt{\frac{\log^{10}(3d)}{n}}
&\leq K'\left(\log(dn)+x^3+\log^{3/2}(dn)+x^2+x^3\right)\sqrt{\frac{\log^{10}(3d)}{n}}\\
&\leq K'\left(2+2\log^{3/2}(dn)+3x^3\right)\sqrt{\frac{\log^{10}(3d)}{n}}\\
&\leq \sqrt 3K'\sqrt{\frac{(4+4\log^3(dn)+9x^6)\log^{10}(3d)}{n}},
}
where the first inequality follows by the elementary inequality $x\log(dn)\leq x^3/3+\log^{3/2}(dn)/(3/2)\leq x^3+\log^{3/2}(dn)$,  the second by $\log(dn)\leq1+\log^{3/2}(dn)$ and $x^2\leq1+x^3$, and the last by $(a+b+c)^2\leq 3(a^2+b^2+c^2)$ for any $a,b,c\in\mathbb R$. 
Consequently, our bound improves \cite{KMB21}'s one when $X_1$ are log-concave. 
Moreover, inspection of the proof of Proposition 1.1 in \cite{FaKo21} leads to the following result, showing that the bound \eqref{eq:lc-md} is sharp if the KLS conjecture is true. 
\begin{proposition}\label{p1}
Let $X=(X_{ij})_{i,j=1}^\infty$ be an array of i.i.d.~random variables such that $\E\exp(c|X_{ij}|)<\infty$ for some $c>0$, $\E X_{ij}=0$, $\E X_{ij}^2=1$ and $\gamma:=\E X_{ij}^3\neq0$. 
Let $W=n^{-1/2}\sum_{i=1}^n X_i$ with $X_i:=(X_{i1},\dots,X_{id})^\top$. 
Suppose that $d$ depends on $n$ so that $(\log^3d)/n\to0$ and $d(\log^3d)/n\to\infty$ as $n\to\infty$. 
Also, let $Z\sim N(0, I_d)$.
Then there exists a sequence $(x_n)$ of positive numbers such that $x_n=o(n^{1/6})$ as $n\to\infty$ and
\[
\limsup_{n\to\infty}\sqrt{\frac{n}{x_n^6+\log^3(dn)}}\left|\frac{P\left(\max_{1\leq j\leq d}W_j> x_n\right)}{P\left(\max_{1\leq j\leq d}Z_j> x_n\right)}-1\right|>0.
\]
\end{proposition}  

\begin{remark}
Cram\'er's original result in the univariate case gives a higher order asymptotic expansion of $P(W>x)$ for moderately large $x$ (see e.g.~\cite[Chapter VIII, Theorem 1]{Pe75}), while \cref{thm:lc-md} concerns only the first order asymptotic expansion of a possible moderate deviation result. 
We refer to a result like \cref{thm:lc-md} as a ``Cram\'er type moderate deviation result'' following the custom in Stein's method literature (see e.g.~Chapter 11 of \cite{CGS11}). 
\end{remark}

Finally, for uniformly log-concave random vectors, we can remove dependence on the constant $\varpi_r$. 
Following \cite{SaWe14}, we define the uniform log-concavity as follows:
\begin{definition}[Uniform log-concavity]
Let $\eps>0$. 
A probability density function $q:\mathbb R^d\to[0,\infty)$ is said to be \textit{$\eps$-uniformly log-concave} if there is a log-concave function $g:\mathbb R^d\to[0,\infty)$ such that $q(x)=g(x)e^{-\eps |x|^2/2}$ for every $x\in\mathbb R^d$. 

A probability measure $\mu$ on $\mathbb R^d$ is said to be $\eps$-uniformly log-concave if it has an $\eps$-uniformly log-concave density. 
A random vector $X$ in $\mathbb R^d$ is said to be $\eps$-uniformly log-concave if its law is $\eps$-uniformly log-concave. 
\end{definition}
Our definition of $\eps$-uniform log-concavity is equivalent to strong log-concavity with variance parameter $\eps^{-1}$ in \cite[Definition 2.9]{SaWe14}. 
Thus, if $q:\mathbb R^d\to[0,\infty)$ is a probability density function of the form $e^{-V}$ with $V:\mathbb R^d\to\mathbb R$ a $C^2$ function, then $q$ is $\eps$-uniformly log-concave if and only if $\Hess V-\eps I_d$ is positive semidefinite; see Proposition 2.24 in \cite{SaWe14}. 

\begin{theorem}\label{thm:ulc-clt}
Let $n\geq2$ be an integer and $\eps\in(0,1)$ a constant. 
Let $X_1,\dots,X_n$ be i.i.d.~isotropic $\eps$-uniformly log-concave random vectors in $\mathbb R^d$, and set $W=n^{-1/2}\sum_{i=1}^nX_i$. 
Let $Z\sim N(0,I_d)$. 
Then, there exist universal constants $c$ and $C$ such that
\ben{\label{ulc:kol}
\sup_{A\in\mathcal R}|P(W\in A)-P(Z\in A)|\leq C\sqrt{\frac{\log^2(dn)\log(2d)}{\eps n}}
}
and
\ben{\label{ulc:wass}
\mcl W_p(W,Z)\leq C\sqrt d\left(\frac{p}{\sqrt{\eps n}}
+\frac{p^{3/2}}{\eps n}\right)
}
for any $p\geq1$. 
Moreover, for 
\ben{\label{ulc:md-ass}
\frac{\log^3(3d)}{\eps n}\leq c,\qquad
0\leq x\leq(\eps n)^{1/6},
}
we have
\ben{\label{ulc:md}
\left|\frac{P(\max_{1\leq j\leq d}W_j>x)}{P(\max_{1\leq j\leq d}Z_j>x)}-1\right|\leq C\frac{(1+x)(\log(dn)+x^2)}{\sqrt{\eps n}}.
}
\end{theorem}
%
For fixed $\eps$, the bounds \eqref{ulc:kol}, \eqref{ulc:wass} and \eqref{ulc:md} are generally rate optimal by the same reasoning as above. In fact, when the coordinates of $X_1$ are i.i.d.~and follows the scaled Weibull distribution with scale parameter 1 and shape parameter $\beta\geq2$, then $X_1$ is $(\beta-1)\{\Gamma(1+2/\beta)-\Gamma(1+1/\beta)^2\}$-uniformly log-concave and its coordinates have non-zero skewness. 
We remark that the bound \eqref{eq:fk} is applicable in the setting of \cref{thm:ulc-clt}, but it leads to an extra $\log n$ factor compared to \eqref{ulc:kol} in the high-dimensional regime $\log n=O(\log d)$. 
Also, regarding the $p$-Wasserstein bound \eqref{ulc:wass}, the dependence on $\eps$ is improved compared to Theorem 3.4 in \cite{Fa19} when $\eps n\geq1$.

The remainder of the paper is organized as follows. 
In \cref{sec:prowass}, we formulate our new coupling inequalities and prove the main results stated in the introduction. 
We prove the coupling inequalities in \cref{sec:proof}. 
\cref{sec:md-max} gives the proof of an auxiliary result to establish the Cram\'er type moderate deviation result. 

\paragraph{Notations.}
For a random vector $\xi$ in $\mathbb R^d$ and $p>0$, we write $\norm{\xi}_p=(\E|\xi|^p)^{1/p}$. 
For a matrix $A$, $\norm{A}_{op}$ and $\norm{A}_{H.S.}$ denote the operator norm and the Hilbert-Schmidt norm of $A$, respectively. 
For two $d\times d$ matrices $A$ and $B$, we write $A\preceq B$ or $B\succeq A$ if $B-A$ is positive semidefinite. We write $\langle A,B\rangle_{H.S.}=\tr(A^\top B)$ for their Hilbert-Schmidt inner product.

\section{Projected Wasserstein bounds}\label{sec:prowass}

The proofs of the main results rely on the following ``projected'' Wasserstein bounds. 
\begin{theorem}\label{thm:lc}
Let $\mu$ be a centered log-concave probability measure on $\mathbb R^d$. 
Suppose that the covariance matrix $\Sigma$ of $\mu$ has rank $r\geq1$. 
Then, for any integer $n\geq1$, we can construct random vectors $W$ and $Z$ in $\mathbb R^d$ such that $W\overset{d}{=}n^{-1/2}\sum_{i=1}^nX_i$ with $X_i\overset{i.i.d.}{\sim}\mu$, $Z\sim N(0,\Sigma)$ and
\ben{\label{lc:prowass}
\|u\cdot(W-Z)\|_p\leq C|\Sigma^{1/2}u|\left(\sqrt{\varpi_r}\frac{p}{\sqrt n}
+\varpi_r\log(2r)\frac{p^{3/2}}{n}
+\frac{1}{\sqrt{\varpi_r\log (2r)}}\frac{p^{5/2}}{n}\right)
}
for all $u\in\mathbb R^d$ and $p\geq1$.
\end{theorem}

\begin{theorem}\label{thm:ulc}
Let $\mu$ be an isotropic probability measure on $\mathbb R^d$. 
Suppose that $\mu$ is $\eps$-uniformly log-concave for some $\eps>0$. 
Then, for any integer $n\geq1$, we can construct random vectors $W$ and $Z$ in $\mathbb R^d$ such that $W\overset{d}{=}n^{-1/2}\sum_{i=1}^nX_i$ with $X_i\overset{i.i.d.}{\sim}\mu$, $Z\sim N(0,I_d)$ and
\ben{\label{eq:ulc}
\|u\cdot(W-Z)\|_p\leq C|u|\left(\frac{p}{\sqrt{\eps n}}+\frac{p^{3/2}}{\eps n}\right)
}
for all $u\in\mathbb R^d$ and $p\geq1$.
\end{theorem}

We prove these theorems in the next section. 
\begin{remark}
It would be worth mentioning that \cref{thm:lc} follows once we prove the corresponding bound for $UW$ with a $d\times d$ orthogonal matrix $U$. This feature allows us to reduce the proof of \cref{thm:lc} to the case $\Sigma=I_d$. By contrast, such reduction is generally impossible if we directly bound the left hand side of \eqref{eq:lc-kol} as in \cite{FaKo21} because the class of rectangles are not rotationally invariant.   
\end{remark}
In the remainder of this section, we prove the main results stated in the introduction using these coupling inequalities. 
Below we will frequently use the inequality $\varpi_r\geq\varpi(N(0,I_r))=1$ without reference. 
\begin{proof}[Proof of \cref{thm:lc-kol}]
For two vectors $x,y\in\mathbb R^d$, we write $x\leq y$ if $x_j\leq y_j$ for all $j=1,\dots,d$. Then we have
\[
\sup_{A\in\mcl R}|P(W\in A)-P(Z\in A)|=\sup_{x\in\mathbb R^{2d}}|P((W^\top,-W^\top)^\top\leq x)-P((Z^\top,-Z^\top)^\top\leq x)|.
\]
Also, the covariance matrix of $(W^\top,-W^\top)^\top$ has rank $r$. Moreover, $(X_i^\top,-X_i^\top)^\top$ are log-concave by Proposition 3.1 in \cite{SaWe14}. Consequently, it suffices to prove
\ben{\label{aim:lc-kol}
\sup_{x\in\mathbb R^{d}}|P(W\leq x)-P(Z\leq x)|\leq C\sqrt{\varpi_r\frac{\log^2(dn)\log(2d)}{n}}
}
when $d\geq2$. Also, since the left hand side is bounded by 1, we may assume 
\ben{\label{wlog:lc-kol}
\varpi_r\frac{\log^2(dn)\log(2d)}{n}\leq1
}
without loss of generality. 

Next, if $\E W_j^2=0$ for some $j$, then $W_j=Z_j=0$ a.s. Hence, with $\mcl J=\{j\in\{1,\dots,d\}:\E W_j^2\neq0\}$, we have
\[
\sup_{x\in\mathbb R^{d}}|P(W\leq x)-P(Z\leq x)|=\sup_{x\in\mathbb R^{d'}}|P((W_j)_{j:\in\mcl J}\leq x)-P((Z_j)_{j\in\mcl J}\leq x)|,
\]
where $d'$ is the number of elements in $\mcl J$. 
Also, the covariance matrix of $(X_{ij})_{j\in\mcl J}$ has rank $r$. 
Moreover, $(X_{ij})_{j\in\mcl J}$ are log-concave by Proposition 3.1 in \cite{SaWe14}. 
Consequently, without loss of generality, we may assume $\mcl J=\{1,\dots,d\}$, i.e.~$\sigma_j=\sqrt{\E W_j^2}>0$ for all $j=1,\dots,d$. 

Again without loss of generality, we may assume that $W$ and $Z$ are the same as in \cref{thm:lc}. 
Fix $x\in\mathbb R^d$ arbitrarily and set 
\[
W^\vee:=\max_{1\leq j\leq d}\frac{W_j-x_j}{\sigma_j},\qquad
Z^\vee:=\max_{1\leq j\leq d}\frac{Z_j-x_j}{\sigma_j}.
\] 
Then we have
\[
P(W\leq x)-P(Z\leq x)=P(W^\vee\leq 0)-P(Z^\vee\leq0).
\]
Let $e_1,\dots,e_d$ be the standard basis of $\mathbb R^d$. For every $j=1,\dots,d$, we apply the bound \eqref{lc:prowass} with $u=e_j/\sigma_j$ and then obtain
\be{
\left\|\frac{W_j-Z_j}{\sigma_j}\right\|_p\leq C\frac{|\Sigma^{1/2}e_j|}{\sigma_j}\left(\sqrt{\varpi_r}\frac{p}{\sqrt n}
+\varpi_r\log(2r)\frac{p^{3/2}}{n}
+\frac{1}{\sqrt{\varpi_r\log (2r)}}\frac{p^{5/2}}{n}\right)
}
for any $p\geq1$. Observe that
\[
\frac{|\Sigma^{1/2}e_j|^2}{\sigma_j^2}
=\frac{e_j^\top\Sigma e_j}{\sigma_j^2}
=1.
\]
Further, let $p=\log(nd)\geq1$. Then
\ba{
\varpi_r\log(2r)\frac{p^{3/2}}{n}
\leq\varpi_r\log(2d)\frac{p^{3/2}}{n}
=\sqrt{\varpi_r\frac{\log^2(dn)}{n}}\sqrt{\varpi_r\frac{\log(dn)\log^2(2d)}{n}}
\leq\sqrt{\varpi_r\frac{\log^2(dn)}{n}},
}
where we used \eqref{wlog:lc-kol} for the last inequality. In addition, note that
\[
\frac{\log^3(dn)}{n}\leq4\frac{\log^3d+\log^3n}{n}\leq4\{1+(3/e)^3\}<36,
\]
where we used \eqref{wlog:lc-kol} and the elementary inequality $\log n\leq (3/e)n^{1/3}$ in the second inequality. Hence we have
\ba{
\frac{\sqrt{\log 2}}{\sqrt{\varpi_r\log (2r)}}\frac{p^{5/2}}{n}
\leq\frac{p^{5/2}}{n}
=\sqrt{\frac{\log^2(dn)}{n}}\sqrt{\frac{\log^3(dn)}{n}}
\leq 6\sqrt{\varpi_r\frac{\log^2(dn)}{n}}.
}
Therefore, there exists a positive universal constant $C_0>0$ such that
\be{
\max_{1\leq j\leq d}\left\|\frac{W_j-Z_j}{\sigma_j}\right\|_p\leq C_0\sqrt{\varpi_r\frac{\log^2(dn)}{n}}.
}
Also, for any $\eta>0$,
\ba{
P(|W^\vee-Z^\vee|>\eta)
\leq\eta^{-p}\E\max_{1\leq j\leq d}\left|\frac{W_j-Z_j}{\sigma_j}\right|^p
\leq\eta^{-p}\sum_{j=1}^d\E\left|\frac{W_j-Z_j}{\sigma_j}\right|^p
\leq d\eta^{-p}\max_{1\leq j\leq d}\E\left|\frac{W_j-Z_j}{\sigma_j}\right|^p.
}
Therefore, taking $\eta=eC_0\sqrt{\varpi_r\log^2(dn)/n}$, we obtain
\[
P(|W^\vee-Z^\vee|>\eta)\leq de^{-p}=\frac{1}{n}.
\]
Thus, by Lemma 2.1 in \cite{CCK16}, 
\[
|P(W\leq x)-P(Z\leq x)|\leq\sup_{t\in\mathbb R}P(|Z^\vee-t|\leq\eta)+\frac{1}{n}.
\]
Observe that
\ba{
P(|Z^\vee-t|\leq\eta)=P(t-\eta\leq Z^\vee\leq t+\eta)=P(Z^\vee\leq (t-\eta)+2\eta)-P(Z^\vee<t-\eta).
}
Thus, by Nazarov's inequality (cf.~\cite{CCK17nazarov}),
\[
\sup_{t\in\mathbb R}P(|Z^\vee-t|\leq\eta)
\leq2\eta(\sqrt{2\log d}+2)
\leq 8\sqrt{\log(2d)}\eta
\leq C\sqrt{\varpi_r\frac{\log^2(dn)\log(2d)}{n}}.
\]
All together, we obtain \eqref{aim:lc-kol}. 
\end{proof}

\begin{proof}[Proof of \cref{thm:lc-wass}]
Without loss of generality, we may assume that $W$ and $Z$ are the same as in \cref{thm:lc}. 
Also, thanks to Jensen's inequality, it suffices to consider the case $p\geq2$. 
Let $e_1,\dots,e_d$ be the standard basis of $\mathbb R^d$. For every $j=1,\dots,d$, we apply the bound \eqref{lc:prowass} with $u=e_j$ and then obtain
\ben{\label{eq:applied}
\|W_j-Z_j\|_p\leq C\sigma_{j}\left(\sqrt{\varpi_r}\frac{p}{\sqrt n}
+\varpi_r\log(2r)\frac{p^{3/2}}{n}
+\frac{1}{\sqrt{\varpi_r\log (2r)}}\frac{p^{5/2}}{n}\right),
}
where we used the identity $|\Sigma^{1/2}e_j|^2=e_j^\top\Sigma e_j=\sigma_{j}^2$. Since 
\ba{
\mcl W_p(W,Z)\leq\|W-Z\|_p\leq\sqrt{\sum_{j=1}^d\|W_j-Z_j\|_p^2},
}
we obtain \eqref{lc:wass}. 

To prove \eqref{lc:wass-simple}, we may assume $p\leq\sqrt n$ without loss of generality. In fact, since $W$ is log-concave by Proposition 3.5 in \cite{SaWe14}, we have by the reverse H\"older inequality (see e.g.~Proposition A.5 in \cite{AGB15})
\ba{
\|W\|_p\leq Cp\E|W|\leq Cp\sqrt{\E|W|^2}=Cp\sqrt{\tr(\Sigma)}. 
}
Also, since $Z$ is Gaussian, we have $\|Z\|_p\leq C\sqrt{p\tr(\Sigma)}$ (cf.~Lemma 6.3 in \cite{FaKo22}). Hence $\|W-Z\|_p\leq Cp\sqrt{\tr(\Sigma)}$. Therefore, if $p>\sqrt n$, the right hand side of \eqref{lc:wass-simple} dominates $C'p\sqrt{\tr(\Sigma)}$, so \eqref{lc:wass-simple} trivially holds with appropriate choice of $C'$. Under the assumptions $p\leq\sqrt n$ and $\varpi_r\log^2(2r)\leq c\sqrt n$, \eqref{lc:wass-simple} immediately follows from \eqref{lc:wass}.
\end{proof}

For the proof of \cref{thm:lc-md}, we use the following general result to derive a Cram\'er type moderate deviation from projected $p$-Wasserstein bounds:
\begin{proposition}\label{md-max}
Let $W$ be a random vector in $\mathbb R^d$ and $Z$ a Gaussian vector in $\mathbb R^d$ with mean 0 and covariance matrix $\Sigma$ such that $\sigma_j>0$ for all $j=1,\dots,d$. Suppose that
\ben{\label{wass-bound-comp}
\max_{1\leq j\leq d}\|W_j-Z_j\|_p\leq Ap^\alpha\Delta\quad\text{for all }1\leq p\leq p_0
}
and
\ben{\label{p0-bound}
\log d+|\log (\Delta/\ul\sigma)|\leq p_0/2
}
with some constants $\alpha\geq0$, $A>0$, $p_0\geq1$ and $\Delta>0$. 
Define $\ol\sigma$ and $\ul\sigma$ as in \eqref{def:sigma}. 
Assume also $\Delta(\log d)^{\alpha+1/2}\leq B\ul\sigma$ for some constant $B>0$. 
Then there exists a positive constant $C$ depending only on $\alpha,A$ and $B$ such that
\ben{\label{eq:mdp-max}
\left|\frac{P(\max_{1\leq j\leq d}W_j>x)}{P(\max_{1\leq j\leq d}Z_j>x)}-1\right|\leq C\left(1+\frac{x}{\ul\sigma}\right)\left(1+\log d+\left|\log\left(\frac{\Delta}{\ul\sigma}\right)\right|+\frac{x^2}{\ol\sigma^2}\right)^{\alpha}\frac{\Delta}{\ul\sigma}
}
for all $0\leq x\leq \min\{\ul\sigma(\Delta/\ul\sigma)^{-1/(2\alpha+1)},\ol\sigma\sqrt{p_0/2}\}$. 
\end{proposition}
The proof of this proposition is given in \cref{sec:md-max}. 
This result can be seen as a multi-dimensional extension of Theorem 2.1 in \cite{FaKo22} in terms of maxima, and it will be of independent interest. See Theorem 4.2 in \cite{FaKo22} for another multi-dimensional extension in terms of Euclidean norms. 
\begin{remark}
In practice, the parameters $A,p_0$ and $\Delta$ in \cref{md-max} will be determined in the following way. 
First, to deduce a meaningful bound from \cref{md-max}, we need to set $\Delta$ to a small value. Then, to make \eqref{p0-bound} hold, we need to take $p_0$ sufficiently large. However, as $p_0$ increases, we need to take $A\Delta$ large enough to make \eqref{wass-bound-comp} hold. The adjustment by $A$ in \eqref{wass-bound-comp} is useful to accomplish the last purpose. 
\end{remark}

\begin{proof}[Proof of \cref{thm:lc-md}]
As in the proof of the previous results, we may assume that $W$ and $Z$ are the same as in \cref{thm:lc}. Then, by \eqref{eq:applied},
\[
\max_{1\leq j\leq d}\|W_j-Z_j\|_p\leq C\ol\sigma\left(\sqrt{\varpi_r}\frac{p}{\sqrt n}
+\varpi_r\log(2r)\frac{p^{3/2}}{n}
+\frac{1}{\sqrt{\varpi_r\log (2r)}}\frac{p^{5/2}}{n}\right)
\]
for any $p\geq1$. Hence, with $\alpha=1$, $p_0=2\min\{n/(\varpi_r\log^2(2r)),n^{1/3}\}$ and $\Delta=\ol\sigma\sqrt{\varpi_r/n}$, we have \eqref{wass-bound-comp} for some universal constant $A$. 
Now assume $\frac{\ol\sigma^2\varpi_r\log^3(3d)}{\ul\sigma^2n}\leq 1$. Then 
\ben{\label{eq:p0-lb}
p_0= 2\min\left\{(n/\varpi_r)^{1/3}\left(n/(\varpi_r\log^3(2r))\right)^{2/3},n^{1/3}\right\}
\geq2(n/\varpi_r)^{1/3}.
}
Since $\Delta/\ul\sigma=\sqrt{\ol\sigma^2\varpi_r/(\ul\sigma^2n)}\leq1$, we have
\ba{
|\log(\Delta/\ul\sigma)|\leq 3(\ul\sigma/\Delta)^{1/3}
\leq 3\left(\frac{n}{\varpi_r}\right)^{1/6}\leq3\left(\frac{n}{\varpi_r}\right)^{1/3}\left(\frac{\ol\sigma^2\varpi_r\log^3(3d)}{\ul\sigma^2n}\right)^{1/6}.
}
Therefore, if
\[
\left(\frac{\ol\sigma^2\varpi_r\log^3(3d)}{\ul\sigma^2n}\right)^{1/6}\leq \frac{1}{6},
\]
then
\[
\log d=n^{1/3}\left(\frac{\log^3d}{n}\right)^{1/3}
\leq \left(\frac{n}{\varpi_r}\right)^{1/3}\left(\frac{\ol\sigma^2\varpi_r\log^3(3d)}{\ul\sigma^2n}\right)^{1/3}
\leq \left(\frac{n}{\varpi_r}\right)^{1/3}\frac{1}{36}\leq\frac{1}{2}\left(\frac{n}{\varpi_r}\right)^{1/3}.
\]
Further, 
\[
|\log(\Delta/\ul\sigma)|
\leq 3\left(\frac{n}{\varpi_r}\right)^{1/3}\frac{1}{6}=\frac{1}{2}\left(\frac{n}{\varpi_r}\right)^{1/3}.
\]
Hence we obtain
\ba{
\log d+|\log (\Delta/\ul\sigma)|\leq \left(\frac{n}{\varpi_r}\right)^{1/3}\leq \frac{p_0}{2},
}
where the last inequality follows by \eqref{eq:p0-lb}. 
In this case we also have 
\ba{
\Delta(\log d)^{3/2}&=\sqrt{\frac{\ol\sigma^2\varpi_r\log^3d}{n}}\leq \ul\sigma
}
and
\ba{
\log d+\left|\log\left(\frac{\Delta}{\ul\sigma}\right)\right|
=\log\left(\frac{d\ul\sigma\sqrt n}{\ol\sigma\sqrt{\varpi_r}}\right)
\leq\log(dn).
}
In addition, 
\[
\ul\sigma\left(\frac{\ul\sigma^2n}{\ol\sigma^2\varpi_r}\right)^{1/6}=\ul\sigma(\Delta/\ul\sigma)^{-1/3}=(\ul\sigma^4/\ol\sigma)^{1/3}(n/\varpi_r)^{1/6}\leq\ol\sigma(n/\varpi_r)^{1/6}\leq\ol\sigma\sqrt{p_0/2},
\] 
where the last inequality follows by \eqref{eq:p0-lb}. 
Combining these estimates and \cref{md-max} with $B=1$ gives the desired result. 
\end{proof}

\begin{proof}[Proof of \cref{thm:ulc-clt}]
Most parts of the proof are almost the same as the corresponding parts of Theorems \ref{thm:lc-kol}--\ref{thm:lc-md}, so we only describe necessary changes. Without loss of generality, we may assume that $W$ and $Z$ are the same as those in \cref{thm:ulc}. We also assume $d>1$ because the results for $d=1$ follow from standard ones. 

First, to prove \eqref{ulc:kol}, we may assume $\log^2(dn)/(n\eps)\leq1$ without loss of generality. Then, with $p=\log(dn)$, we have $\max_{j}\|W_j-Z_j\|_p\leq C_0p/\sqrt{n\eps}$ for some universal constant $C_0>0$. Now, fix $A=\prod_{j=1}^d[a_j,b_j]\in\mcl R$ and set
\[
W^\vee:=\max_{1\leq j\leq d}\{(a_j-W_j)\vee(W_j-b_j)\},\qquad
Z^\vee:=\max_{1\leq j\leq d}\{(a_j-Z_j)\vee(Z_j-b_j)\}.
\]
Then we have
\[
|P(W\in A)-P(Z\in A)|=|P(W^\vee\leq0)-P(Z^\vee\leq0)|\leq\sup_{t\in\mathbb R}P(|Z^\vee-t|\leq\eta)+P(|W^\vee-Z^\vee|>\eta)
\]
for any $\eta>0$ by Lemma 2.1 in \cite{CCK16}. Since $P(|W^\vee-Z^\vee|>\eta)\leq d\eta^{-p}\max_{j}\E|W_j-Z_j|^p$, we can deduce \eqref{ulc:kol} similarly to the proof of \cref{thm:lc} by choosing $\eta=eC_0\log(dn)/\sqrt{n\eps}$. 

Next, the proof of \eqref{ulc:wass} is a straightforward modification of that of \cref{thm:lc-wass}. 

Finally, to prove \eqref{ulc:md}, we may assume $\eps n\geq1$ because we take $c=1/2$ for \eqref{ulc:md-ass}. Observe that $\ol\sigma=\ul\sigma=1$ in the present setting. Then, we can verify that \eqref{wass-bound-comp} and \eqref{p0-bound} hold with $\alpha=1,p_0=2\eps n,\Delta=1/\sqrt{\eps n}$ and some universal constant $A>0$. Also, the condition $\Delta(\log d)^{3/2}\leq B$ holds with $B=\sqrt c$ by \eqref{ulc:md-ass}. Thus, \eqref{ulc:md} follows from \cref{md-max}.
\end{proof}

\section{Proof of Theorems \ref{thm:lc} and \ref{thm:ulc}}\label{sec:proof}

Given a probability density function $q$ on $\mathbb R^d$, we write $\Cov(q)=\Cov(X)$, where $X$ is a random vector in $\mathbb R^d$ with density $q$. 

%
%
%

\subsection{Score bound}

The proof of \cref{thm:ulc} uses an analog of the so-called Stein kernel method to get a $p$-Wasserstein bound. This method relies on a score-based bound for the Wasserstein distance due to \cite{OV00} (see e.g.~Eq.(3.8) of \cite{LNP15}), so we first develop a projected Wasserstein version of this bound. For later use, we develop such a bound for Markov kernels. 
Since \cite{OV00}'s proof is not constructive, it causes a measurability issue when applied to Markov kernels. To avoid this difficulty, we construct an explicit coupling using the so-called \textit{F\"ollmer process}. We refer to \cite{ElMi20} for background of the F\"ollmer process. 

We fix a standard Gaussian vector $G$ in $\mathbb R^d$ independent of everything else. 
For a random vector $W$ in $\mathbb R^d$ and $t\in[0,1)$, we set $W[t]:=\sqrt t W+\sqrt{1-t}G$. 
It is straightforward to check that the law of $W[t]$ has a smooth density $f_{W[t]}$ with respect to $N(0,I_d)$. Moreover, $f_{W[t]}$ is strictly positive by Lemma 3.1 of \cite{JoSu01}. Therefore, we can define the \textit{score} of $W[t]$ with respect to $N(0,I_d)$ by $\rho_{W[t]}(w)=\nabla \log f_{W[t]}(w)$, $w\in\mathbb{R}^d$.

\begin{proposition}\label{score-bound}
Let $\mathcal P$ be a Markov kernel from a measurable space $(\mcl X,\mcl A)$ to $\mathbb R^d$. 
Suppose that $\mcl P(x,\cdot)$ has a smooth density $f^x$ with respect to $N(0,I_d)$ for all $x\in\mcl X$. 
Then, there exists a Markov kernel $\mcl Q$ from $(\mcl X,\mcl A)$ to $\mathbb R^d\times\mathbb R^d$ satisfying the following conditions for any $x\in\mcl X$:
\begin{enumerate}[label=(\roman*)] 

\item For any Borel set $A$ in $\mathbb R^d$, $\mcl Q(x,A\times\mathbb R^d)=\mcl P(x,A)$ and $\mcl Q(x,\mathbb R^d\times A)=N(0,I_d)(A)$.

\item If $W$ and $Z$ are random vectors in $\mathbb R^d$ such that $(W,Z)\sim\mcl Q(x,\cdot)$, then
\ben{\label{eq:score-bound}
\|u\cdot(W-Z)\|_p\leq\int_0^1 \frac{1}{\sqrt t}\|u\cdot\rho_{W[t]}(W[t])\|_pdt
} 
for any $p\geq1$ and $u\in\mathbb R^d$.

\end{enumerate}
\end{proposition}

\begin{proof}
Consider a filtered probability space $(\Omega,\mcl F,\mathbf F=(\mcl F)_{t\in[0,1]},Q)$ on which a $d$-dimensional standard $\mathbf F$-Brownian motion $Y=(Y_t)_{t\in[0,1]}$ is defined. 
For every $x\in\mcl X$, define a process $M^x=(M^x_t)_{t\in[0,1]}$ as $M^x_t=P_{1-t}f^x(Y_t)$, where $P_{1-t}f^x(y)=\E f^x(y+\sqrt{1-t}G)$, $y\in\mathbb R^d$. Then, define a measure $P^x$ on $(\Omega,\mcl F)$ as $P^x(F)=\E_Q[M^x_11_F]$, i.e.~$dP^x/dQ=M^x_1$ (cf.~Eq.(19) in \cite{ElLe18}). Also, define a process $B^x=(B^x_t)_{t\in[0,1]}$ as 
\ben{\label{eq:follmer}
B^x_t=Y_t-\int_0^t\nabla\log (P_{1-s}f^x)(Y_s)ds.
}
By Theorem 2.1 in \cite{ElLe18}, $P^x$ is a probability measure on $(\Omega,\mcl F)$, and $B^x$ is well-defined and a $d$-dimensional standard $\mathbf F$-Brownian motion under $P^x$. Moreover, the law of $Y_1$ has density $f^x$ with respect to $N(0,I_d)$ under $P^x$. 
We then define $\mcl Q(x,A)=P^x((Y_1,B^x_1)\in A)=\E_Q[1_A(Y_1,B^x_1)M^x_1]$ for any measurable set $A\subset\mathbb R^d\times\mathbb R^d$. It is evident from Fubini's theorem that $\mcl Q$ is a Markov kernel from $(\mcl X,\mcl A)$ to $\mathbb R^d\times\mathbb R^d$. 
Also, $\mcl Q$ satisfies condition (i) by construction. So it remains to check condition (ii) is satisfied. 

Let $W$ and $Z$ be as in condition (ii). We have
\ba{
\norm{u\cdot(W-Z)}_p=\norm{u\cdot(Y_1-B^x_1)}_{L^p(P^x)}
\leq\int_0^1\norm{u\cdot\nabla\log (P_{1-t}f^x)(Y_t)}_{L^p(P^x)}dt,
}
where the last inequality follows from \eqref{eq:follmer} and the integral Minkowski inequality (see e.g.~Proposition C.4 in \cite{Ja97}). 
For every $t\in[0,1]$, the law of $Y_t$ under $P^x$ is the same as the law of $\sqrt tW[t]=tW+\sqrt{t(1-t)}G$. This is pointed out by \cite[Eq.(7)]{ElMi20}, but we can directly prove it as follows. 
For any bounded measurable function $h:\mathbb R^d\to\mathbb R$, we have 
\[
\E_{P^x}[h(Y_t)]=\E_Q[h(Y_t)M_1^x]=\E_Q[h(Y_t)f^x(Y_1)].
\] 
Let $G'\sim N(0,I_d)$ be independent of $G$. Recall that $Y$ is a standard Brownian motion in $\mathbb R^d$ under $Q$. Then, one can easily check that $(Y_t,Y_1)$ has the same law as $(tG'+\sqrt{t(1-t)}G,G')$ under $Q$. Hence
\[
\E_{P^x}[h(Y_t)]=\E[h(tG'+\sqrt{t(1-t)}G)f^x(G')]=\E[h(tW+\sqrt{t(1-t)}G)],
\] 
where the last equality holds because $f^x$ is the density of $\mcl P(x,\cdot)$ with respect to $N(0,I_d)$. This proves the desired result. Therefore, 
\ba{
\norm{u\cdot(W-Z)}_p
\leq\int_0^1\norm{u\cdot\nabla\log (P_{1-t}f^x)(\sqrt tW[t])}_pdt.
}
Now, for any bounded measurable function $h:\mathbb R^d\to\mathbb R$, we have by definition
\ba{
\E h(W[t])=\E[h(\sqrt tG'+\sqrt{1-t}G)f^x(G')].
}
Since $(\sqrt tG'+\sqrt{1-t}G,G')$ has the same law as $(G',\sqrt tG'+\sqrt{1-t}G)$, we obtain
\ba{
\E h(W[t])=\E[h(G')f^x(\sqrt tG'+\sqrt{1-t}G)]=\E[h(G')P_{1-t}f^x(\sqrt tG')].
}
This implies that $w\mapsto P_{1-t}f^x(\sqrt tw)$ is the density of the law of $W[t]$ with respect to $N(0,I_d)$. 
Consequently, $\rho_{W[t]}(w)=\sqrt t\nabla\log(P_{1-t}f^x)(\sqrt tw)$. Hence  
\ba{
\|u\cdot\nabla \log (P_{1-t}f^x)(\sqrt tW[t])\|_p
=\|u\cdot\rho_{W[t]}(W[t])\|_p/\sqrt t.
}
All together, we obtain \eqref{eq:score-bound}. 
\end{proof}

\subsection{Stein kernel bound}

To bound the right hand side of \eqref{eq:score-bound}, we use the notion of Stein kernel. Throughout this subsection, $\mu$ denotes a centered probability measure on $\mathbb R^d$. 
\begin{definition}[Stein kernel]
A $d\times d$ matrix valued measurable function $\tau$ on $\mathbb R^d$ is called a \textit{Stein kernel} for $\mu$ if all entries of $\tau$ belong to $L^1(\mu)$ and
\ben{\label{eq:sk}
\int x\cdot h(x)\mu(dx)=\int\langle\tau(x),\nabla h(x)\rangle_{H.S.}\mu(dx)
}
for every compactly supported smooth function $h:\mathbb R^d\to\mathbb R^d$. 

We say that $\tau$ is a Stein kernel for a centered random vector $W$ in $\mathbb R^d$ if it is a Stein kernel for the law of $W$. In this case, \eqref{eq:sk} reads as
\[
\E[W\cdot h(W)]=\E[\langle \tau(W),\nabla h(W)\rangle_{H.S.}].
\]
\end{definition}


\begin{remark}
The definition of Stein kernels here is the same as in \cite{CFP19}. The same definition is also adopted in \cite{NPS14,Fa19} and \cite{MiSh21}. As discussed in the introduction of \cite{CFP19}, some articles use a slightly weaker version of \eqref{eq:sk} to define Stein kernels; see e.g.~\cite{LNP15} and \cite{FaKo21}. We need the present stronger version to apply Lemma 2.9 in \cite{NPS14}. 
\end{remark}

We obtain the following bound by a direct extension of the arguments in the proof of \cite[Proposition 3.4]{LNP15}:
\begin{lemma}\label{sk-bound}
Let $W$ be a centered random vector in $\mathbb R^d$. Suppose that $W$ has a Stein kernel $\tau$. Then
\ben{\label{eq:sk-bound}
\int_0^1 \frac{1}{\sqrt t}\|u\cdot\rho_{W[t]}(W[t])\|_pdt\leq C\sqrt p\|(\tau(W)^\top-I_d) u\|_{p}
}
for all $u\in\mathbb R^d$ and $p\geq1$. 
\end{lemma}

\begin{proof}
By Lemma 2.9 in \cite{NPS14}, 
\[
\rho_{W[t]}(W[t])=\frac{t}{\sqrt{1-t}}\E[(\tau(W)-I_d)G\mid W[t]]\quad\text{a.s.}
\]
for all $t\in[0,1)$. Hence, by Jensen's inequality,
\ba{
\int_0^1 \frac{1}{\sqrt t}\|u\cdot\rho_{W[t]}(W[t])\|_pdt
&\leq\int_0^1 \sqrt{\frac{t}{1-t}}\|u\cdot(\tau(W)-I_d)G\|_pdt.
}
Conditional on $W$, $u\cdot(\tau(W)-I_d)G$ follows the normal distribution with mean 0 and variance $u^\top(\tau(W)-I_d)(\tau(W)-I_d)^\top u=|(\tau(W)^\top-I_d) u|^2$. 
Moreover, when $\zeta$ is a normal variable with mean 0 and variance $\sigma^2$, then $\|\zeta\|_p=2\sigma\{\Gamma((p+1)/2)/\sqrt \pi\}^{1/p}\leq C\sqrt p\sigma$, where the upper bound follows by Stirling's formula. Consequently, we obtain
\ba{
\int_0^1 \frac{1}{\sqrt t}\|u\cdot\rho_{W[t]}(W[t])\|_pdt
&\leq C\sqrt p\|(\tau(W)^\top-I_d) u\|_p.
}
This completes the proof. 
\end{proof}

To get the desired bounds from \cref{sk-bound}, we need to construct a Stein kernel $\tau$ for $\mu$ such that $\int|\tau(x)^\top u|^p\mu(dx)$ enjoys a dimension-free bound (up to the constant $\varpi_d$) for any $u\in\mathbb R^d$ and $p\geq1$ when $\mu$ is log-concave. If we additionally assume that $\mu$ is uniformly log-concave, such a construction is given by \cite{Fa19}; see Corollary 2.4 ibidem. However, it is unclear whether this construction gives an appropriate bound in the log-concave case: Only an entry-wise bound is available in this case. 
Recently, \cite{MiSh21} developed another construction that would admit a bound for $\int\|\tau(x)\|_{op}^p\mu(dx)$; see Theorem 1.5 and the proof of Theorem 5.7 ibidem. Their original bound is not dimension-free, but it implicitly depends on $\varpi_d$, so it is improvable if the KLS conjecture is true. However, inspection of their proof suggests that the bound would still contain a poly-log factor on the dimension even if we assume the KLS conjecture. 

Here, our first key observation is that a Stein kernel constructed in \cite{CFP19} enjoys a dimension-free bound for $\int|\tau(x)^\top u|^2\mu(dx)$, provided that the KLS conjecture is true. 
We first recall the construction of \cite{CFP19}. 
For every integer $k\geq1$, we define $W^{1,2}_k(\mu)$ as the closure of the set of all compactly supported smooth functions $h:\mathbb R^d\to\mathbb R^k$ in $L^2(\mu)$, with respect to the norm $\sqrt{\int(|h|^2+\norm{\nabla h}_{H.S.}^2)d\mu}$. 
We regard $W^{1,2}_k(\mu)$ as a Hilbert space equipped with this norm. 
We also write $W^{1,2}_{k,0}(\mu)$ for the set of functions $h\in W^{1,2}_{k}(\mu)$ with $\int hd\mu=0$. 
If $\tau$ is a Stein kernel for $\mu$ and $\int(|x|^2+\|\tau(x)\|_{H.S.}^2)\mu(dx)<\infty$, it is evident by definition that \eqref{eq:sk} holds for any $h\in W^{1,2}_d(\mu)$.  
\begin{theorem}[\cite{CFP19}, Theorem 2.4]\label{thm:cfp}
If $\varpi(\mu)<\infty$, there exists a unique function $\psi_\mu\in W^{1,2}_{d,0}(\mu)$ such that $\tau_\mu:=\nabla\psi_\mu$ is a Stein kernel for $\mu$. 
\end{theorem}

\begin{lemma}\label{key0}
Under the assumptions of \cref{thm:cfp}, if $X\sim\mu$, then
\[
\E|\tau_\mu(X)^\top u|^2\leq\varpi(\mu)\E|X\cdot u|^2
\]
for all $u\in\mathbb R^d$.
\end{lemma}

\begin{proof}
Observe that $\varpi(\mu)<\infty$ implies $\E|X|^2<\infty$. Hence, we can apply \eqref{eq:sk} with $\tau=\tau_\mu$ and $h(x)=uu^\top \psi_\mu(x)$, which yields
\ba{
\E[X\cdot uu^\top \psi_\mu(X)]
&=\E[\langle\tau_\mu(X),uu^\top \tau_\mu(X)\rangle_{H.S.}]\\
&=\E[\tr(\tau_\mu(X)^\top uu^\top \tau_\mu(X))]
=\E|\tau_\mu(X)^\top u|^2.
}
Hence, by the Cauchy--Schwarz inequality,
\ba{
\E|\tau_\mu(X)^\top u|^2\leq\sqrt{\E|X\cdot u|^2\E|u^\top\psi_\mu(X)|^2}.
}
By a standard approximation argument, one can easily verify that \eqref{eq:poincare} with $\varpi=\varpi(\mu)$ holds for any $h\in W^{1,2}_1(\mu)$. Applying this inequality with $h(x)=u^\top\psi_\mu(x)$, we obtain
\ba{
\E|u^\top\psi_\mu(X)|^2\leq\varpi(\mu)\E|u^\top\tau_\mu(X)|^2=\varpi(\mu)\E|\tau_\mu(X)^\top u|^2.
}
Consequently, 
\[
\E|\tau_\mu(X)^\top u|^2\leq\sqrt{\varpi(\mu)\E|X\cdot u|^2}\sqrt{\E|\tau_\mu(X)^\top u|^2}.
\]
This yields the desired result. 
\end{proof}

A drawback of \cite{CFP19}'s construction is that it is based on the Lax--Milgram theorem and thus implicit. So it is generally difficult to control moments higher than two. 
Fortunately, this is not the case when $\mu$ is uniformly log-concave: We can find an explicit representation of the function $\psi_\mu$ in \cref{thm:cfp}, which allows us to get a dimension-free bound for $\norm{\tau_\mu}_{op}$:
\begin{lemma}\label{key1}
Let $\mu$ be a centered probability measure on $\mathbb R^d$. 
Suppose that $\mu$ has a smooth, positive and $\eps$-uniformly log-concave density for some $\eps>0$.  Then the function $\tau_\mu$ in \cref{thm:cfp} satisfies $\norm{\tau_\mu}_{op}\leq\eps^{-1}$ $\mu$-a.s. 
\end{lemma}

\begin{proof}
By assumption, the density of $\mu$ is of the form $e^{-V}$ with $V:\mathbb R^d\to\mathbb R$ a $C^\infty$ function such that $\Hess V\succeq\eps I_d$. 
For every $x\in\mathbb R^d$, consider the following stochastic differential equation (SDE):
\ben{\label{eq:langevin}
X^x_0=x,\quad
dX^x_t=-\nabla V(X^x_t)dt+\sqrt 2dB_t,~t\geq0.
}
This SDE has a unique strong solution. To see this, note that $-x\cdot\nabla V(x)\leq V(0)-V(x)$ for all $x\in\mathbb R^d$ because $V$ is convex. Hence, by Theorem 5.1 in \cite{SaWe14}, there exists a constant $c>0$ such that $-x\cdot\nabla V(x)\leq c$ for all $x\in\mathbb R^d$. Also, note that $\nabla V$ is locally Lipschitz. Therefore, the SDE \eqref{eq:langevin} has a unique strong solution by Theorems 3.7 and 3.11 in \cite[Chapter 5]{EtKu86}. 

Let $X^x=(X^x_t)_{t\geq0}$ be the solution to \eqref{eq:langevin}. 
Below we show that the function $x\mapsto\psi_t(x):=\int_0^t\E X_s^xds$ converges to some function $\psi$ in the space $W^{1,2}_d(\mu)$ as $t\to\infty$, and $\psi=\psi_\mu$ $\mu$-a.s. Indeed, this fact follows from the general result of \cite[Theorem 5.10]{ArHo19} based on Dirichlet form theory (see also \cite[Proposition 3.1]{ArHo22}), but we give a proof without referring to Dirichlet forms for readers' convenience. 

Let $(T_t)_{t\geq0}$ be the transition semigroup on $L^2(\mu)$ associated with the SDE \eqref{eq:langevin}. 
It is well-known that its generator is given by 
\[
L=-\nabla V\cdot\nabla+\nabla\cdot\nabla,
\]
where $\nabla\cdot\nabla$ denotes the Laplacian. 
Using integration by parts, we have for any compactly supported smooth functions $g:\mathbb R^d\to\mathbb R$ and $h:\mathbb R^d\to\mathbb R$
\ben{\label{eq:IBP}
\int_{\mathbb R^d}gLhd\mu
=-\int_{\mathbb R^d}\nabla g\cdot\nabla hd\mu.
}
From this identity we obtain $\domain(L)\subset W^{1,2}_1(\mu)$. Also, by definition, \eqref{eq:IBP} also holds for any $g\in W^{1,2}_1(\mu)$ and $h\in\domain(L)$. 
Further, by Proposition 9.2 in \cite[Chapter 4]{EtKu86}, \eqref{eq:IBP} implies that $\mu$ is a stationary distribution for $(T_t)_{t\geq0}$. 

Next we show that 
\ben{\label{var-decay}
\int|T_tg|^2d\mu\leq e^{-2t/\varpi(\mu)}\int g^2d\mu
}
for any $t\geq0$ and $g\in L^2(\mu)$ with $\int gd\mu=0$. Since $\domain(L)$ is dense in $L^2(\mu)$, it suffices to prove \eqref{var-decay} when $g\in\domain(L)$. Then, by Proposition 1.5(b) in \cite[Chapter 1]{EtKu86},
\ben{\label{eq:fp}
\frac{d}{dt}T_tg=LT_tg.
}
Therefore, 
\ba{
\frac{d}{dt}\int|T_tg|^2d\mu=2\int T_tg\cdot LT_tgd\mu=-2\int |\nabla T_tg|^2d\mu,
}
where the first identity follows from \eqref{eq:fp} and the second from \eqref{eq:IBP}. Since $\int T_tgd\mu=\int gd\mu=0$, we have
\[
\int |T_tg|^2d\mu\leq\varpi(\mu)\int |\nabla T_tg|^2d\mu.
\]
Consequently, 
\[
\frac{d}{dt}\int|T_tg|^2d\mu\leq-\frac{2}{\varpi(\mu)}\int |T_tg|^2d\mu.
\]
Thus, we obtain \eqref{var-decay} by Gronwall's inequality. 

Applying \eqref{var-decay} with $g(x)=x_j$ for every $j=1,\dots,d$, we obtain
\ben{\label{var-decay-applied}
\int|\E X_t^x|^2\mu(dx)\leq e^{-2t/\varpi(\mu)}\int |x|^2\mu(dx).
}
Hence, for $0<s<t$, 
\ba{
\int|\psi_t-\psi_s|^2d\mu
&\leq\int(t-s)\left(\int_s^t|\E X^x_u|^2du\right)\mu(dx)\\
&\leq\frac{\varpi(\mu)(t-s)(e^{-2s/\varpi(\mu)}-e^{-2t/\varpi(\mu)})}{2}\int |x|^2\mu(dx).
}
Thus $\int|\psi_t-\psi_s|^2d\mu\to0$ as $s,t\to\infty$. 
Meanwhile, by Theorem 39 in \cite[Chapter V]{Pr05}, we can take a version of $X^x$ such that the map $x\mapsto X_t^x$ is differentiable for any $t\geq0$ and its derivative $\nabla_xX_t^x=(\partial_{x_1}X_t^x,\dots,\partial_{x_d}X_t^x)$ satisfies
\[
\nabla_xX^x_t=I_d-\int_0^t\Hess V(X^x_s)\nabla_xX^x_sds.
\]
From this equation we can prove
\ben{\label{dx-op-bound}
\|\nabla_xX^x_t\|_{op}\leq e^{-\eps t}.
}
To see this, fix $u\in\mathbb R^d$ and set $D_t=\nabla_xX^x_tu$. Then we have
\ba{
\frac{d}{dt}|D_t|^2=2D_t\cdot\frac{d}{dt}D_t=-2D_t\Hess V(X^x_t)D_t\leq-2\eps|D_t|^2,
}
where the last inequality follows from $\Hess V\succeq\eps I_d$. Hence we obtain $|D_t|^2\leq|D_0|^2e^{-2\eps t}=|u|^2e^{-2\eps t}$ by Gronwall's inequality. This implies \eqref{dx-op-bound}. 
By \eqref{dx-op-bound}, $\psi_t$ is differentiable and $\nabla\psi_t=\int_0^t\E\nabla_xX^x_sds$. Hence, for $0<s<t$,
\ba{
\int\|\nabla\psi_t-\nabla\psi_s\|_{H.S.}^2d\mu
&\leq\int\left(\int_s^t\|\E\nabla_x X^x_u\|_{H.S.}du\right)^2\mu(dx)\\
&\leq d\left(\int_s^te^{-\eps u}du\right)^2\leq\frac{d(e^{-\eps s}-e^{-\eps t})^2}{\eps^2}.
}
So $\int\|\nabla\psi_t-\nabla\psi_s\|_{H.S.}^2d\mu\to0$ as $s,t\to\infty$. 
Consequently, $\psi_t$ converges to some function $\psi$ in the space $W^{1,2}_d(\mu)$ as $t\to\infty$. 

To prove $\psi=\psi_\mu$ $\mu$-a.s., we need to check $\int\psi d\mu=0$ and $\nabla\psi$ is a Stein kernel for $\mu$. The former is immediate because $\int\E X_s^x \mu(dx)=\int x\mu(dx)=0$, where the first identity holds because $\mu$ is a stationary distribution for $(T_t)_{t\geq0}$. 
Meanwhile, by Proposition 1.5(a) in \cite[Chapter 1]{EtKu86},
\[
\E X^x_t-x=L\psi_t(x)\quad\mu\text{-a.s.},
\]
where the operator $L$ is applied to $\psi_t$ coordinate-wise. 
Thus, by \eqref{eq:IBP}, for any $g\in W^{1,2}_d(\mu)$,
\[
\int g(x)\cdot(\E X^x_t-x)\mu(dx)=-\int\nabla \psi_t\cdot\nabla gd\mu.
\]
By \eqref{var-decay-applied}, $\int|\E X^x_t|^2\mu(dx)\to0$ as $t\to\infty$. Therefore, letting $t\to\infty$ in the above identity, we obtain
\[
\int g(x)\cdot x\mu(dx)=\int\nabla \psi\cdot\nabla g d\mu.
\]
So $\nabla\psi$ is a Stein kernel for $\mu$. 
All together, $\psi=\psi_\mu$ $\mu$-a.s. 

Finally, since $\|\nabla\psi_t\|_{op}\leq\int_0^te^{-\eps s}ds$ for all $t\geq0$, we have $\norm{\nabla\psi}_{op}\leq\eps^{-1}$ $\mu$-a.s. This completes the proof. 
\end{proof}

\begin{remark}
(a) The operator norm bound as in \cref{key1} holds for the other constructions mentioned above; see \cite[Corollary 2.4]{Fa19} and \cite[Theorem 3.1(1-a)]{MiSh21}. 
What is important for the construction by \cite{CFP19} is that it also satisfies the estimate in \cref{key0}, i.e.~a (nicely) dimension-free bound for $\E|\tau_\mu(X)^\top u|^2$ in the general log-concave case. 
This type of estimate plays a key role in the proof of \cref{ulc-bound} below because we need a precise estimate for the variance of $\tau_\mu(X)^\top u$ to derive an appropriate bound from an application of Rosenthal's inequality. 
As discussed above, such an estimate is currently unavailable for other constructions.  

\noindent(b) \cite{Fa19}'s construction is known to ensure the positive definiteness of the Stein kernel. This property is important in some applications; see \cite{FaMi22}. It is unclear whether the construction by \cite{CFP19} always has this property. 
The proof of \cref{key1} implies that this happens when $\int_0^t\E[\nabla_xX_s^x]ds$ is positive definite for sufficiently large $t$, but verification of this condition is not straightforward. 

\noindent(c) As we mentioned in the proof of \cref{key1}, the explicit representation of $\tau_\mu$ is already given in \cite{ArHo19}; see Remark 5.11(ii) ibidem. Indeed, the bound $\|\tau_\mu\|_{op}\leq\eps^{-1}$ also follows from this representation and gradient bounds for diffusion semigroups, e.g.~Proposition 3.2.5 of \cite{BGL14}.  
\end{remark}

Combining these observations give the following upper bound for \eqref{eq:sk-bound}:
\begin{proposition}\label{ulc-bound}
Let $X_1,\dots,X_n$ be centered independent random vectors in $\mathbb R^d$. 
Suppose that $X_i$ has a smooth, positive and $\eps$-uniformly log-concave density for all $i=1,\dots,n$ and some $\eps>0$. 
Then, for any $d\times d$ matrix $A$, $W:=n^{-1/2}\sum_{i=1}^nAX_i$ has a Stein kernel $\tau$ satisfying
\ben{\label{eq:rosenthal}
\|(\tau(W)^\top-I_d) u\|_p\leq \frac{C\norm{A}_{op}}{n}\left(\sqrt{p\sum_{i=1}^n\varpi(X_i)u^\top\Sigma_i u}
+\frac{p}{\eps}|A^\top u|\right)
+\left|\left(\frac{1}{n}\sum_{i=1}^n\Sigma_i-I_d\right)u\right|
}
for all $u\in\mathbb R^d$ and $p\geq1$, where $\Sigma_i=A\Cov(X_i)A^\top$.  
\end{proposition}

\begin{proof}
For every $i=1,\dots,n$, let $\tau_i$ be the Stein kernel for $X_i$ given by \cref{thm:cfp}. By \cref{key1}, $\|\tau_i(X_i)\|_{op}\leq\eps^{-1}$ a.s. Let
\[
\tau(w):=\frac{1}{n}\E\left[\sum_{i=1}^nA\tau_i(X_i)A^\top\mid W=w\right].
\]
It is straightforward to check that $\tau$ is a Stein kernel for $W$. 
We show that this $\tau$ satisfies \eqref{eq:rosenthal}. Thanks to Jensen's inequality, it suffices to prove \eqref{eq:rosenthal} when $p\geq2$. 
Since $\E\tau_i(X_i)=\Cov(X_i)$ for every $i$ by the definition of Stein kernel, we have
\[
\|(\tau(W)^\top-I_d) u\|_p
\leq\|(\tau(W)^\top-\E\tau(W)^\top) u\|_p+\left|\left(\frac{1}{n}\sum_{i=1}^n\Sigma_i-I_d\right)u\right|.
\]
By Jensen's inequality,
\ba{
\|(\tau(W)^\top-\E\tau(W)^\top) u\|_p
&\leq\frac{1}{n}\lnorm{\sum_{i=1}^nA\{\tau_i(X_i)^\top-\E\tau_i(X_i)^\top\}A^\top u}_p\\
&\leq\frac{\norm{A}_{op}}{n}\lnorm{\sum_{i=1}^n\{\tau_i(X_i)^\top-\E\tau_i(X_i)^\top\}A^\top u}_p.
}
Applying Rosenthal's inequality for random vectors in $\mathbb R^d$ (cf.~Eq.(4.2) in \cite{Pi94}), we obtain 
\ba{
&\|(\tau(W)^\top-\E\tau(W)^\top) u\|_p\\
&\leq \frac{C\norm{A}_{op}}{n}\left(\sqrt{p\sum_{i=1}^n\E|\{\tau_i(X_i)^\top-\E\tau_i(X_i)^\top\}A^\top u|^2}
+p\left\|\max_{1\leq i\leq n}|\{\tau_i(X_i)^\top-\E\tau_i(X_i)^\top\}A^\top u|\right\|_p\right)\\
&\leq \frac{C\norm{A}_{op}}{n}\left(\sqrt{p\sum_{i=1}^n4\varpi(X_i)\E|A^\top u\cdot X_i|^2}
+2p\eps^{-1}|A^\top u|\right),
}
where we used \cref{key0}, $\norm{\tau_i(X_i)}_{op}\leq\eps^{-1}$ a.s.~and Jensen's inequality to get the last line. Since $\E|A^\top u\cdot X_i|^2=(A^\top u)^\top\Cov(X_i)A^\top u=u^\top\Sigma_iu$, we complete the proof. 
\end{proof}

%
%
%

The results obtained so far will be used in the following form:
\begin{proposition}\label{prop:sk-wrap}
Let $A$ be a $d\times d$ matrix and $(f_y)_{y\in\mathbb R^d}$ be a family of smooth, positive and $\eps$-uniformly log-concave densities on $\mathbb R^d$ with some constant $\eps>0$. Suppose that the map $\mathbb R^d\times\mathbb R^d\ni(y,x)\mapsto f_y(x)\in[0,\infty)$ is measurable. 
For every $y\in\mathbb R^d$, let $\nu_{y}$ be the law of $\xi-\E\xi$ with $\xi$ a random vector in $\mathbb R^d$ having density $f_y$. 
Then, there exists a Markov kernel $\mcl Q$ from $(\mathbb R^d)^n$ to $\mathbb R^d\times\mathbb R^d$ satisfying the following conditions for any $\bs y=(y^{(i)})_{i=1}^n\in(\mathbb R^d)^n$:
\begin{enumerate}[label=(\roman*)]

\item $\mcl Q(\bs y,\cdot\times\mathbb R^d)$ equals the law of $An^{-1/2}\sum_{i=1}^n \xi_i$, where $\xi_1,\dots,\xi_n$ are independent random vectors in $\mathbb R^d$ such that $\mcl L(\xi_i)=\nu_{y^{(i)}}$ for all $i=1,\dots,n$.

\item $\mcl Q(\bs y,\mathbb R^d\times\cdot)$ is the $d$-dimensional standard normal distribution.

\item
If $W'$ and $Z'$ are random vectors in $\mathbb R^d$ such that $(W',Z')\sim\mcl Q(\bs y,\cdot)$, then
\ba{
\norm{u\cdot(W'-Z')}_p
&\leq \frac{C\|A\|_{op}}{n}\left(p\sqrt{\sum_{i=1}^n\varpi(\nu_{y^{(i)}})u^\top A\Cov(f_{y^{(i)}})A^\top u}
+\frac{p^{3/2}}{\eps}|A^\top u|\right)\\
&\qquad+C\sqrt p\left|\left(\frac{1}{n}\sum_{i=1}^nA\Cov(f_{y^{(i)}})A^\top-I_d\right)u\right|
}
for any $p\geq1$ and $u\in\mathbb R^d$. 

\end{enumerate}
\end{proposition}

\begin{proof}
For every $\bs y=(y^{(i)})_{i=1}^n\in(\mathbb R^d)^n$ and every Borel set $E\subset\mathbb R^d$, define $\mcl P(\bs y,E)=P(An^{-1/2}\sum_{i=1}^n \xi_i\in E)$, where $\xi_1,\dots,\xi_n$ are the same as in condition (i). Then, $\mcl P$ defines a Markov kernel from $(\mathbb R^d)^n$ to $\mathbb R^d$. Applying \cref{score-bound} to $\mcl P$, we can construct a Markov kernel $\mcl Q$ from $(\mathbb R^d)^n$ to $\mathbb R^d\times\mathbb R^d$ satisfying conditions (i) and (ii) for all $\bs y\in(\mathbb R^d)^n$ and
\[
\|u\cdot(W'-Z')\|_p\leq\int_0^1 \frac{1}{\sqrt t}\|u\cdot\rho_{W'[t]}(W'[t])\|_pdt
\]
with $(W',Z')\sim\mcl Q(\bs y,\cdot)$ for any $p\geq1$ and $u\in\mathbb R^d$. 
By \cref{ulc-bound}, $W'$ has a Stein kernel $\tau$ satisfying
\ba{
&\|(\tau(W)^\top-I_d) u\|_p\\
&\leq \frac{C\norm{A}_{op}}{n}\left(\sqrt{p\sum_{i=1}^n\varpi(\nu_{y^{(i)}})u^\top A\Cov(f_{y^{(i)}})A^\top u}
+\frac{p}{\eps}|A^\top u|\right)
+\left|\left(\frac{1}{n}\sum_{i=1}^nA\Cov(f_{y^{(i)}})A^\top-I_d\right)u\right|.
}
Combining these bounds with \cref{sk-bound} shows that condition (iii) is satisfied. 
\end{proof}

\subsection{Proof of \cref{thm:ulc}}

First we prove the claim when $\mu$ has a smooth, positive density $f$. 
In this case, by Proposition \ref{prop:sk-wrap} with $A=I_d$ and $f_y\equiv f$, we can construct random vectors $W$ and $Z$ in $\mathbb R^d$ such that $W\overset{d}{=}n^{-1/2}\sum_{i=1}^nX_i$ with $X_i\overset{i.i.d.}{\sim}\mu$, $Z\sim N(0,I_d)$ and
\ba{
\|u\cdot(W-Z)\|_p\leq \frac{C|u|\sqrt p}{n}\left(\sqrt{pn\varpi(\mu)}
+\frac{p}{\eps}\right)
}
for all $u\in\mathbb R^d$ and $p\geq1$. By the Brascamp--Lieb inequality (see e.g.~Proposition 10.1 in \cite{SaWe14}), we have $\varpi(\mu)\leq\eps^{-1}$. Hence \eqref{eq:ulc} holds. 

For the general case, take a constant $a\in(0,1)$ arbitrarily, and let $\mu^a$ be the law of the random vector $\sqrt{1-a}X+\sqrt{a} G$, where $X\sim\mu$ and $G\sim N(0,I_d)$ are independent. Clearly, $\mu^a$ is isotropic and has a smooth, positive density. Also, by Theorem 3.7 in \cite{SaWe14}, $\mu^a$ is $\eps/(1-(1-\eps)a)$-uniformly log-concave. Hence we can construct random vectors $W^a$ and $Z^a$ in $\mathbb R^d$ such that $W^a\overset{d}{=}n^{-1/2}\sum_{i=1}^nX_i$ with $X_i\overset{i.i.d.}{\sim}\mu^a$, $Z^a\sim N(0,I_d)$ and
\be{
\|u\cdot(W^a-Z^a)\|_p\leq C|u|\left(\frac{p}{\sqrt{\eps n}}+\frac{p^{3/2}}{\eps n}\right)
}
for all $u\in\mathbb R^d$ and $p\geq1$. In particular, the family $\{(W^a,Z^a):a\in(0,1)\}$ is tight, so by Prohorov's theorem there exists a sequence $(a_k)_{k=1}^\infty$ of numbers in $(0,1)$ such that $a_k\to0$ and $(W^{a_k},Z^{a_k})$ converges in law to some pair $(W,Z)$ of random vectors in $\mathbb R^d$ as $k\to\infty$. It is clear that $W\overset{d}{=}n^{-1/2}\sum_{i=1}^nX_i$ with $X_i\overset{i.i.d.}{\sim}\mu$ and $Z\sim N(0,I_d)$. Also, by Theorem 3.4 in \cite{Bi99},
\be{
\|u\cdot(W-Z)\|_p\leq\liminf_{k\to\infty}\|u\cdot(W^{a_k}-Z^{a_k})\|_p\leq C|u|\left(\frac{p}{\sqrt{\eps n}}+\frac{p^{3/2}}{\eps n}\right)
}
for all $u\in\mathbb R^d$ and $p\geq1$. So $W$ and $Z$ are desired ones. \qed

\subsection{Stochastic localization}

A naive idea to use \cref{ulc-bound} for the proof of \cref{thm:lc} is to approximate a log-concave distribution by a uniformly log-concave one. To be precise, given a log-concave random vector $X$ in $\mathbb R^d$ and a positive constant $\eps$, we wish to construct an $\eps$-uniformly log-concave random vector $X^\eps$ such that $\norm{u\cdot(X-X^\eps)}_p=O(\eps)$ as $\eps\downarrow0$ for all $u\in\mathbb R^d$ and $p\geq2$. If this is possible, it is not difficult to see that we can obtain a bound of order $1/\sqrt n$ for \eqref{eq:sk-bound} via \cref{ulc-bound}. 
However, this approach seems hopeless because the currently best known bound for $\mcl W_p(X,X_\eps)$ is presumably the one given by Proposition 1 in \cite{DKRD21} and it is of order $\eps^{1/p}$. 
Instead, we take an alternative idea of using \textit{Eldan's stochastic localization} that enables us to express a log-concave distribution as a mixture of uniformly log-concave distributions and some ``nice'' distribution. The latter can be handled by the martingale embedding method developed in \cite{EMZ20}. Below we detail this strategy.  


Throughout this subsection, we assume $d\geq2$. 
Let $\mu$ be an isotropic log-concave probability measure on $\mathbb R^d$. 
Suppose that $\mu$ has a smooth, positive density $f$ with respect to $N(0,I_d)$. 
Consider the following SDE:
\ben{\label{eq:follmer2}
Y_0=0,\qquad
dY_t=\nabla \log (P_{1-t}f)(Y_t)dt+dB_t,\quad t\in[0,1],
}
where $B_t$ is a $d$-dimensional Brownian motion. 
By Theorem 2.1 in \cite{ElLe18}, this SDE has a weak solution $Y=(Y_t)_{t\in[0,1]}$ such that $Y_1\sim\mu$.  
Moreover, we can show that the solution to \eqref{eq:follmer2} is unique in law; see \cref{thm:unique} in \cref{sec:appendix}. 
$Y$ is known as the F\"ollmer process in the literature. 

For $t\in(0,1)$ and $y\in\mathbb R^d$, define a probability density function $f_{t,y}:\mathbb R^d\to[0,\infty)$ as
\[
f_{t,y}(x)=\frac{f(x)\phi_{tI_d}(x-y)}{P_{t}f(y)}=\frac{f(x)e^{-|x-y|^2/(2t)}}{\int_{\mathbb R^d}f(z)e^{-|z-y|^2/(2t)}dz},\quad x\in\mathbb R^d,
\]
where $\phi_{tI_d}$ is the (Lebesgue) density of $N(0,tI_d)$. 
For every $t\in(0,1)$, the conditional law of $Y_1$ given $\mcl F_t:=\sigma(Y_s:0\leq s\leq t)$ has density $f_{1-t,Y_t}$. 
In fact, by Eq.(21) in \cite{ElLe18}, there exists a probability measure $Q$ such that $Y$ is a $d$-dimensional standard Brownian motion under $Q$ and 
\[
P(Y_1\in A\mid\mcl F_t)=\E_Q\left[1_A(Y_1)\frac{f(Y_1)}{P_{1-t}f(Y_t)}\mid\mcl F_t\right]
\]
for every Borel set $A\subset\mathbb R^d$. Under $Q$, the conditional law of $Y_1$ given $\mcl F_t$ is $N(Y_t,(1-t)I_d)$. Consequently, 
\ba{
P(Y_1\in A\mid\mcl F_t)=\int_{\mathbb R^d}1_A(x)\frac{f(x)}{P_{1-t}f(Y_t)}\phi_{(1-t)I_d}(x-Y_t)dx
=\int_Af_{1-t,Y_t}(x)dx.
}
So the desired result follows. 
In particular, we have $\E[Y_1\mid\mcl F_t]=m(t,Y_t)$, where
\[
m(t,y):=\int_{\mathbb R^d}xf_{1-t,y}(x)dx.
\]

Since $\mu$ is log-concave, $\Hess(\log f)\preceq I_d$; hence $\Hess(\log f_{1-t,y})\preceq -\frac{t}{1-t}I_d$ for all $t\in(0,1)$ and $y\in\mathbb R^d$. This means that $f_{1-t,y}$ is $t/(1-t)$-uniformly log-concave. 
Thus, conditional on $\mcl F_t$, $Y_1-m(t,Y_t)$ is centered and $t/(1-t)$-uniformly log-concave, so the sum of its independent copies would be coupled with a suitable Gaussian vector by \cref{prop:sk-wrap}.
Therefore, if we can couple the sum of independent copies of $m_t:=m(t,Y_t)$ with a suitable Gaussian vector for a moderately small $t$, the proof of \cref{thm:lc} will be complete. 
In this section, we accomplish this by a version of the martingale embedding method of \cite{EMZ20}. 


Let $(Y^{(1)},B^{(1)}),\dots,(Y^{(n)},B^{(n)})$ be independent copies of $(Y,B)$. By definition, $Y^{(i)}$ satisfies the following SDE for every $i=1,\dots,n$:
\be{
Y^{(i)}_0=0,\qquad
dY^{(i)}_t=\nabla \log (P_{1-t}f)(Y^{(i)}_t)dt+dB^{(i)}_t,\quad t\in[0,1].
}
Define $m^{(i)}_t:=m(t,Y^{(i)}_t)$.

\begin{proposition}\label{me-bound}
There exists a universal constant $c_0\geq3$ such that, for any $0\leq \eps\leq(c_0\varpi_d\log (2d))^{-1}$, we can construct a random vector $Z_\eps\sim N(0,\Cov(m_\eps))$ on the same probability space where $Y^{(1)},\dots,Y^{(n)}$ are defined and such that
\ba{
\left\|u\cdot\left(\frac{1}{\sqrt n}\sum_{i=1}^nm^{(i)}_\eps-Z_\eps\right)\right\|_p
\leq C|u|\sqrt \eps\left(\frac{p}{\sqrt{n}}+\frac{p^{5/2}}{n}\right)
}
for all $u\in\mathbb R^d$ and $p\geq1$. 
\end{proposition}

For the proof of \cref{me-bound}, we use the following stochastic integral representation of $m_t$:
\[
m_t=\int_0^t\Gamma_sdB_s,\quad\text{where}\quad\Gamma_t=\frac{\Cov(f_{1-t,Y_t})}{1-t}. 
\]
This representation is found in \cite[Section 4]{EMZ20}, but we can directly verify it using Ito's formula; see Lemma 13 in \cite{ELS20} for details. 

We collect some properties of $\Gamma_t$ necessary for our proof. The first one is a lower bound of $\E\Gamma_t$:
\begin{lemma}[\cite{ElMi20}, Corollary 3]\label{gamma-lb}
If $t\leq\frac{1}{2\varpi(\mu)+1}$, then $\E\Gamma_t\succeq I_d/3$.
\end{lemma}

To get moment bounds for $\norm{\Gamma_t}_{op}$, we employ a recent result of \cite{KlLe22}. For this purpose, we relate our notation to theirs. 
For $t>0$ and $\theta\in\mathbb R^d$, we define a probability density function $p_{t,\theta}:\mathbb R^d\to[0,\infty)$ as
\[
p_{t,\theta}(x)=f_{1/(1+t),\theta/(1+t)}(x)=\frac{f(x)e^{-|x|^2/2}e^{\theta\cdot x-t|x|^2/2}}{\int_{\mathbb R^d}f(z)e^{-|z|^2/2}e^{\theta\cdot z-t|z|^2/2}dz},\quad x\in\mathbb R^d.
\]
Note that this is the same notation as in \cite{KlLe22} because $f$ is the density of $\mu$ with respect to $N(0,I_d)$. 
Then, as in \cite{KlLe22}, set
\[
a(t,\theta)=\int_{\mathbb R^d}xp_{t,\theta}(x)dx,\qquad
A(t,\theta)=\Cov(p_{t,\theta}).
\]
Also, define 
\ba{
\tilde B_t=\int_0^{t/(1+t)}\frac{1}{1-s}dB_s,\qquad
\theta_t=(1+t)Y_{t/(1+t)},\quad t\geq0.
}

\begin{lemma}\label{lem:theta}
$(\tilde B_t)_{t\geq0}$ is a standard Brownian motion in $\mathbb R^d$. Moreover, $(\theta_t)_{t\geq0}$ satisfies the following SDE:
\ben{\label{eq:follmer3}
d\theta_t=a(t,\theta_t)dt+d\tilde B_t,\qquad\theta_0=0.
}
In addition, the solution to \eqref{eq:follmer3} is unique in law.  
\end{lemma}

\begin{proof}
This fact is pointed out in \cite[Section 4.2]{KlPu21}. We give a formal proof for the sake of completeness. 
First, we can easily check that the quadratic covariation matrix process of $(\tilde B_t)_{t\geq0}$ is given by $[\tilde B,\tilde B]_t=tI_d$; hence, $(\tilde B_t)_{t\geq0}$ is a standard Brownian motion in $\mathbb R^d$ by L\'evy's characterization (see e.g.~Theorem 40 in \cite[Chapter II]{Pr05}). 
Next, by a direct calculation, we have for any $t\in(0,1)$ and $y\in\mathbb R^d$
\[
\nabla \log (P_{1-t}f)(y)=\frac{m(t,y)-y}{1-t}.
\]
Hence
\[
Y_t-B_t=\int_0^t\frac{m(s,Y_s)-Y_s}{1-s}ds.
\]
Also, integration by parts gives
\[
\frac{Y_t}{1-t}=\int_0^t\frac{1}{1-s}dY_s+\int_0^t\frac{Y_s}{(1-s)^2}ds.
\]
Consequently, for any $t\geq0$,
\ba{
\tilde B_t&=\int_0^{t/(1+t)}\frac{1}{1-s}dY_s-\int_0^{t/(1+t)}\frac{m(s,Y_s)-Y_s}{(1-s)^2}ds\\
&=\theta_t-\int_0^{t/(1+t)}\frac{m(s,Y_s)}{(1-s)^2}ds
=\theta_t-\int_0^{t}m(u/(1+u),Y_{u/(1+u)})du.
}
Since $a(u,\theta)=m(u/(1+u),\theta/(1+u))$ for $u>0$ and $\theta\in\mathbb R^d$ by definition, $(\theta_t)_{t\geq0}$ satisfies \eqref{eq:follmer3}.

Conversely, if $(\theta_t)_{t\geq0}$ is a solution to \eqref{eq:follmer3} with a standard Brownian motion $(\tilde B_t)_{t\geq0}$ in $\mathbb R^d$, in a similar manner to the above, we can verify that
\[
B_t=\int_0^{t/(1-t)}\frac{1}{1+s}d\tilde B_s,\qquad t\in[0,1),
\]
is a standard Brownian motion in $\mathbb R^d$ and $Y_t=(1-t)\theta_{t/(1-t)}$ satisfies \eqref{eq:follmer2} for $t\in[0,1)$. Hence uniqueness in law for \eqref{eq:follmer3} follows from that for \eqref{eq:follmer2}. 
\end{proof}

Thanks to \cref{lem:theta}, the process $(\theta_t)_{t\geq0}$ has the same law as the one defined in \cite{KlLe22}; see Eq.(18) ibidem. 

Let
\[
\kappa_d:=\sup_{\mu\in\LC_d}\left\|\int_{\mathbb R^d}x_1xx^\top\mu(dx)\right\|_{H.S.}.
\]
By Fact 6.1 in \cite{El13}, there exists a positive universal constant $C_0>0$ such that 
\ben{\label{kappa-bound}
\kappa_d^2\leq C_0\varpi_d.
}

\begin{lemma}\label{lem:kl}
There exist positive universal constants $C$ and $c$ such that
\[
\norm{\norm{\Gamma_s}_{op}}_p\leq Cp
\]
for any $0< s\leq(1/2)\wedge(c \kappa_d^2 \cdot\log d)^{-1}$. 
\end{lemma}

\begin{proof}
Let $A_t=A(t,\theta_t)$. By the proof of \cite[Corollary 5.4]{KlLe22}, there exist positive universal constants $C$ and $c$ such that
\[
\E\norm{A_t}_{op}^p\leq 2^p+C^pp!
\]
for any $0<t\leq(c \kappa_d^2 \cdot\log d)^{-1}$. 
Next, recall that $A_t=\Cov(f_{1/(1+t),Y_{t/(1+t)}})$. Hence, for any $s\in(0,1)$, $A_{s/(1-s)}=\Cov(f_{1-s,Y_{s}})$. 
Therefore, for $0<s\leq(1/2)\wedge(2c\kappa_d\cdot\log d)^{-1}$,
\ba{
\E\norm{\Gamma_s}_{op}^p
=\frac{\E\norm{A_{s/(1-s)}}_{op}^p}{(1-s)^p}
\leq4^p+(2C)^pp!
\leq4^p+(2Cp)^p.
}
So we obtain the desired result. 
\end{proof}


We will also need the notion of matrix geometric mean. 
For two positive definite matrices $A$ and $B$, their geometric mean is defined as
\[
A\#B:=\sqrt A\sqrt{A^{-1/2}BA^{-1/2}}\sqrt A.
\]
$A\# B$ is evidently positive definite. Also, we have $A\# B=B\# A$ by Theorem 4.1.3 in \cite{Bh07}. 
The matrix geometric mean is useful because of the following lemma. 
\begin{lemma}\label{lem:mgm-est}
Let $A$ and $B$ be two $d\times d$ positive definite matrices. Then
\[
A+B-2A\#B\preceq (A-B)A^{-1}(A-B).
\]
\end{lemma}

\begin{proof}
Using the definition of $A\# B$, we obtain
\ba{
A+B-2A\#B
&=\sqrt A\left(I_d+A^{-1/2}BA^{-1/2}-2\sqrt{A^{-1/2}BA^{-1/2}}\right)\sqrt A\\
&=\sqrt A\left(I_d-\sqrt{A^{-1/2}BA^{-1/2}}\right)^2\sqrt A\\
&\preceq\sqrt A\left(I_d-A^{-1/2}BA^{-1/2}\right)^2\sqrt A\\
&=(\sqrt A-BA^{-1/2})(\sqrt A-A^{-1/2}B)
=(A-B)A^{-1}(A-B).
}
\end{proof}

\begin{proof}[Proof of \cref{me-bound}]
Recall that we assume $d\geq2$, so $\log (2d)>1$. Also, note that $\varpi_d\geq1$ because $\varpi(N(0,I_d))=1$. 
Let $C_0$ and $c$ be the universal constants in \eqref{kappa-bound} and \cref{lem:kl}, respectively. Then, we take $c_0:=\max\{cC_0,3\}$. By construction, we have $c_0\varpi_d\log (2d)\geq\max\{c\kappa_d^2\log d,2\varpi_d+1\}$. 

For every $i=1,\dots,n$, we can write
\[
m^{(i)}_t=\int_0^t\Gamma^{(i)}_sdB^{(i)}_s,\quad\text{where}\quad\Gamma^{(i)}_t=\frac{\Cov\left(f_{1-t,Y^{(i)}_t}\right)}{1-t}. 
\]
Consider a continuous local martingale $M=(M_t)_{t\in[0,1]}$ in $\mathbb R^d$ defined as 
\[
M_t=\frac{1}{\sqrt n}\sum_{i=1}^nm^{(i)}_t
=\frac{1}{\sqrt n}\sum_{i=1}^n\int_0^t\Gamma_s^{(i)}dB^{(i)}_s,\qquad t\in[0,1].
\]
The quadratic covariation matrix process of $M$ is given by
\ba{
[M,M]_t=\frac{1}{n}\sum_{i=1}^n\int_0^t(\Gamma_s^{(i)})^2ds
=\int_0^t\bar\Gamma_s^2ds,
}
where
\[
\bar\Gamma_s:=\sqrt{\frac{1}{n}\sum_{i=1}^n(\Gamma_s^{(i)})^2}.
\]
Since $\bar\Gamma_s$ is invertible and satisfies $\bar\Gamma_s^{-1}\bar\Gamma_s^2\bar\Gamma_s^{-1}=I_d$, we can define a process $\tilde B=(\tilde B_t)_{t\in[0,1]}$ as
\[
\tilde B_t=\int_0^t\bar\Gamma_s^{-1}dM_s,\qquad t\in[0,1].
\]
We evidently have $M_t=\int_0^t\bar\Gamma_sd\tilde B_s$. 
Moreover, since $[\tilde B,\tilde B]_t=tI_d$ for all $t\in[0,1]$, $\tilde B$ is a standard Brownian motion in $\mathbb R^d$ by L\'evy's characterization. 

Next, since $0\preceq\E[(\Gamma_t-\E\Gamma_t)^2]=\E\Gamma_t^2-(\E\Gamma_t)^2$, we have by \cref{gamma-lb}
\ben{\label{eq:gamma-lb}
\E\Gamma_t^2\succeq\frac{1}{9}I_d\quad\text{for any}\quad 0\leq t\leq \eps.
}
In particular, $\E\Gamma_t^2$ is invertible for $0\leq t\leq \eps$. 
Define
\[
U_t=\sqrt{\E[\Gamma_t^2]^{-1/2}\bar\Gamma_t^2\E[\Gamma_t^2]^{-1/2}}\sqrt{\E[\Gamma_t^2]}\bar\Gamma_t^{-1}.
\] 
We can easily check that $U_t^\top U_t=U_tU_t^\top=I_d$ and $\E[\Gamma_t^2]\#\bar\Gamma_t^2=\sqrt{\E[\Gamma_t^2]}U_t\bar\Gamma_t$. 
In addition, we define a process $\hat B=(\hat B_t)_{t\in[0,\eps]}$ as
\[
\hat B_t=\int_0^tU_sd\tilde B_s,\qquad t\in[0,\eps].
\] 
We have for all $t\in[0,\eps]$
\[
[\hat B,\hat B]_t=\int_0^tU_sU_s^\top ds=tI_d.
\]
Therefore, by L\'evy's characterization again, $\hat B$ is a standard Brownian motion in $\mathbb R^d$. Hence
\[
Z_\eps:=\int_0^\eps\sqrt{\E[\Gamma_t^2]}d\hat B_t
=\int_0^\eps\sqrt{\E[\Gamma_t^2]}U_td\tilde B_t
\]
defines a centered Gaussian vector in $\mathbb R^d$ such that
\[
\Cov(Z_\eps)=\int_0^\eps\E[\Gamma_t^2]dt=\Cov(m_\eps).
\]

We are going to bound $\|u\cdot(M_\eps-Z_\eps)\|_p$. Thanks to Jensen's inequality, it suffices to consider the case $p\geq2$. Since
\[
u\cdot(M_\eps-Z_\eps)=\int_0^\eps u^\top(\bar\Gamma_t-\sqrt{\E[\Gamma_t^2]}U_t)d\tilde B_t,
\]
we obtain by Proposition 4.2 in \cite{BY82}
\ba{
&\|u\cdot(M_\eps-Z_\eps)\|_p\\
&\leq C\sqrt p\left\|\sqrt{\int_0^\eps u^\top(\bar\Gamma_t-\sqrt{\E[\Gamma_t^2]}U_t)(\bar\Gamma_t-U_t^\top\sqrt{\E[\Gamma_t^2]})udt}\right\|_p\\
&=C\sqrt p\left\|\int_0^\eps u^\top(\bar\Gamma_t^2+\E[\Gamma_t^2]-\sqrt{\E[\Gamma_t^2]}U_t\bar\Gamma_t-\Gamma_tU_t^\top\sqrt{\E[\Gamma_t^2]})udt\right\|_{p/2}^{1/2}.
}
By construction, $\sqrt{\E[\Gamma_t^2]}U_t\bar\Gamma_t=\E[\Gamma_t^2]\#\bar\Gamma_t^2$, which is a (random) symmetric matrix. Hence
\ben{\label{first-bound}
\|u\cdot(M_\eps-Z_\eps)\|_p
\leq C\sqrt p\left\|\int_0^\eps u^\top(\bar\Gamma_t^2+\E[\Gamma_t^2]-2\E[\Gamma_t^2]\#\bar\Gamma_t^2)udt\right\|_{p/2}^{1/2}.
}
Therefore, we obtain by \cref{lem:mgm-est}
\ba{
\|u\cdot(M_\eps-Z_\eps)\|_p
&\leq C\sqrt p\left\|\int_0^\eps u^\top(\bar\Gamma_t^2-\E[\Gamma_t^2])\E[\Gamma_t^2]^{-1}(\bar\Gamma_t^2-\E[\Gamma_t^2])udt\right\|_{p/2}^{1/2}.
}
Then, by \eqref{eq:gamma-lb},
\ba{
\|u\cdot(M_\eps-Z_\eps)\|_p
&\leq C\sqrt p\left\|\int_0^\eps |(\bar\Gamma_t^2-\E[\Gamma_t^2])u|^2dt\right\|_{p/2}^{1/2}.
}
By the integral Minkowski inequality, we obtain
\[
\|u\cdot(M_\eps-Z_\eps)\|_p
\leq C\sqrt{p\int_0^\eps \|(\bar\Gamma_t^2-\E[\Gamma_t^2])u\|_p^2dt}.
\]
To evaluate the integrand, observe that
\ba{
(\bar\Gamma_t^2-\E[\Gamma_t^2])u
&=\frac{1}{n}\sum_{i=1}^n\left((\Gamma_t^{(i)})^2u-\E[(\Gamma_t^{(i)})^2u]\right).
}
Hence, noting that sub-exponential tails are equivalent to linear growth of $L^r$-norms (cf.~\cite[Proposition 2.7.1]{Ve18}), we have by Lemma 2.1 in \cite{FaKo22} and \cref{lem:kl}
\ba{
\|(\bar\Gamma_t^2-\E[\Gamma_t^2])u\|_p
\leq \frac{C|u|}{n}(\sqrt{pn}+p^2)
=C|u|\left(\sqrt{\frac{p}{n}}+\frac{p^2}{n}\right).
}
Consequently, 
\[
\|u\cdot(M_\eps-Z_\eps)\|_p
\leq C|u|\sqrt \eps\left(\frac{p}{\sqrt{n}}+\frac{p^{5/2}}{n}\right).
\]
This completes the proof.
\end{proof}

\subsection{Proof of \cref{thm:lc}}

We will use the following simple fact about the Poincar\'e constant. 
\begin{lemma}\label{lem:lin-po}
Let $X$ be a random vector in $\mathbb R^d$. Then, for any $k\times d$ matrix $A$, $\varpi(AX)\leq\norm{A}_{op}^2\varpi(X)$.
\end{lemma}

\begin{proof}
Fix a locally Lipschitz function $h:\mathbb R^k\to\mathbb R$ arbitrarily. Define a function $\tilde h:\mathbb R^d\to\mathbb R$ as $\tilde h(x)=h(Ax)$ for $x\in\mathbb R^d$. It is straightforward to check that $\tilde h$ is locally Lipschitz and satisfies $|\nabla\tilde h(x)|\leq \norm{A}_{op}|\nabla h(Ax)|$ for all $x\in\mathbb R^d$. Hence, we have $\Var[h(AX)]=\Var[\tilde h(X)]\leq\varpi(X)\norm{A}_{op}^2\E|\nabla h(AX)|^2$ by the definition of $\varpi(X)$. This implies the desired result. 
\end{proof}


\begin{proof}[Proof of \cref{thm:lc}]
We divide the proof into two steps.
\medskip

\noindent
\textbf{Step 1.} First we prove the result when $\Sigma=I_d$.
Since $\mu\times\mu$ is log-concave by Proposition 3.2 in \cite{SaWe14}, the result for $d=1$ follows from that for $d=2$. Hence we may assume $d\geq2$. 

By a similar argument as in the proof of \cref{thm:ulc}, we may assume that $\mu$ has a positive and smooth density. 
Let $c_0$ be the universal constant in \cref{me-bound}, and set $\eps:=(c_0\varpi_d\log (2d))^{-1}$ and $\Sigma_\eps:=\E\Cov(f_{1-\eps,Y_\eps})$. Note that $\Sigma_\eps\succeq(1-\eps)\E\Gamma_\eps\succeq\frac{2}{9}I_d$ by \cref{gamma-lb}. In particular, $\Sigma_\eps$ is invertible. 
Also, for every $y\in\mathbb R^d$, let $\nu_{\eps,y}$ be the law of $\xi-\E\xi$ with $\xi$ a random vector in $\mathbb R^d$ having density $f_{1-\eps,y}$. 
Then, by Proposition \ref{prop:sk-wrap} with $A=\Sigma_\eps^{-1/2}$ and $f_y=f_{1-\eps,y}$, there exists a Markov kernel $\mcl Q$ from $(\mathbb R^d)^n$ to $\mathbb R^d\times\mathbb R^d$ satisfying the following conditions for any $\bs y=(y^{(i)})_{i=1}^n\in(\mathbb R^d)^n$:
\begin{enumerate}[label=(\roman*)]

\item $\mcl Q(\bs y,\cdot\times\mathbb R^d)$ equals the law of $\Sigma_\eps^{-1/2}n^{-1/2}\sum_{i=1}^n \xi_i$, where $\xi_1,\dots,\xi_n$ are independent random vectors in $\mathbb R^d$ such that $\mcl L(\xi_i)=\nu_{\eps,y^{(i)}}$ for all $i=1,\dots,n$.

\item $\mcl Q(\bs y,\mathbb R^d\times\cdot)$ is the $d$-dimensional standard normal distribution.

\item\label{sk-bound-y} If $W'$ and $Z'$ are random vectors in $\mathbb R^d$ such that $(W',Z')\sim\mcl Q(\bs y,\cdot)$, then
\ba{
&\norm{u\cdot(W'-Z')}_p\\
&\leq \frac{C\|\Sigma_\eps^{-1/2}\|_{op}}{n}\left(p\sqrt{\sum_{i=1}^n\varpi(\nu_{\eps,y^{(i)}})u^\top\Sigma_\eps^{-1/2}\Cov(f_{1-\eps,y^{(i)}})\Sigma_\eps^{-1/2}u}
+\frac{p^{3/2}}{\eps}|\Sigma_\eps^{-1/2}u|\right)\\
&\qquad+\sqrt p\left|\left(\frac{1}{n}\sum_{i=1}^n\Sigma_\eps^{-1/2}\Cov(f_{1-\eps,y^{(i)}})\Sigma_\eps^{-1/2}-I_d\right)u\right|
}
for any $p\geq1$ and $u\in\mathbb R^d$.

\end{enumerate}
Note that $\varpi(\nu_{\eps,y})\leq\varpi_d\norm{\Cov(f_{1-\eps,y})}_{op}$ for any $y\in\mathbb R^d$ by \cref{lem:lin-po}.
We will use the bound in property \ref{sk-bound-y} after applying this inequality. 
Next, let $Z_\eps$ be as in \cref{me-bound}. 
Define a probability distribution $\Pi$ on $\mathbb R^d\times\mathbb R^d$ as
\ba{
\Pi(A)=\E\left[\int_{\mathbb R^d\times\mathbb R^d} 1_A\left(\frac{1}{\sqrt n}\sum_{i=1}^nm^{(i)}_\eps+\sqrt{\Sigma_\eps}w,Z_\eps+\sqrt{\Sigma_\eps}z\right)\mcl Q((Y^{(1)}_\eps,\dots,Y^{(n)}_\eps),dwdz)\right].
}
Take random vectors $W$ and $Z$ in $\mathbb R^d$ such that $(W,Z)\sim\Pi$. Below we show that these $W$ and $Z$ are desired ones. 

Recall that $Y^{(1)},\dots,Y^{(n)}$ are independent copies of $Y$, where $Y$ is defined as \eqref{eq:follmer2}. 
Moreover, recall that $Y_1\sim\mu$, $m_\eps=\E[Y_1\mid\mcl F_\eps]$ and the conditional law of $Y_1$ given $\mcl F_\eps$ has density $f_{1-\eps,Y_\eps}$. Hence $W\overset{d}{=}n^{-1/2}\sum_{i=1}^nX_i$ with $X_i\overset{i.i.d.}{\sim}\mu$ by construction. 
Also, it is straightforward to see that $Z\sim N(0,\Cov(Z_\eps)+\Sigma_\eps)$. 
Then, since $\E(Y_1-m_\eps)m_\eps^\top=\E[\E[Y_1-m_\eps\mid\mcl F_\eps]m_\eps^\top]=0$, we have
\ba{
I_d=\E Y_1Y_1^\top
=\E(Y_1-m_\eps)(Y_1-m_\eps)^\top+\E m_\eps m_\eps^\top.
}
By construction, we have $\E m_\eps m_\eps^\top=\Cov(m_\eps)=\Cov(Z_\eps)$. Also, 
\ba{
\E(Y_1-m_\eps)(Y_1-m_\eps)^\top
=\E[\Cov(Y_1\mid\mcl F_\eps)]
=\E\Cov(f_{1-\eps,Y_\eps})=\Sigma_\eps.
}
Therefore, $I_d=\Sigma_\eps+\Cov(Z_\eps)$, so $Z\sim N(0,I_d)$. 
Finally, for any $p\geq1$ and $u\in\mathbb R^d$,
\ba{
&\norm{u\cdot(W-Z)}_p\\
&=\left(\E\left[\int_{\mathbb R^d\times\mathbb R^d} \left|u\cdot\left(\frac{1}{\sqrt n}\sum_{i=1}^nm^{(i)}_\eps+\sqrt{\Sigma_\eps}w-\left(Z_\eps+\sqrt{\Sigma_\eps}z\right)\right)\right|^p\mcl Q((Y^{(1)}_\eps,\dots,Y^{(n)}_\eps),dwdz)\right]\right)^{1/p}\\
&\leq\left\|u\cdot\left(\frac{1}{\sqrt n}\sum_{i=1}^nm^{(i)}_\eps-Z_\eps\right)\right\|_p
+\left(\E\left[\int_{\mathbb R^d\times\mathbb R^d} \left|\sqrt{\Sigma_\eps}u\cdot\left(w-z\right)\right|^p\mcl Q((Y^{(1)}_\eps,\dots,Y^{(n)}_\eps),dwdz)\right]\right)^{1/p}\\
&\leq C|u|\sqrt \eps\left(\frac{p}{\sqrt{n}}+\frac{p^{5/2}}{n}\right)
+\frac{C\|\Sigma_\eps^{-1/2}\|_{op}|u|}{n}\left(p\sqrt{\varpi_d}\left\|\sqrt{\sum_{i=1}^n\norm{\Cov(f_{1-\eps,Y_\eps^{(i)}})}_{op}^2}\right\|_p
+\frac{p^{3/2}}{\eps}\right)\\
&\quad+\sqrt p\left\|\left(\frac{1}{n}\sum_{i=1}^n\Sigma_\eps^{-1/2}\Cov(f_{1-\eps,Y_\eps^{(i)}})-\sqrt{\Sigma_\eps}\right)u\right\|_p,
}
where the last line follows by \cref{me-bound} and property \ref{sk-bound-y}. 
Now, recall that $f_{1-\eps,y}$ is $\eps/(1-\eps)$-uniformly log-concave. Hence, by the Brascamp--Lieb inequality, 
\[
\norm{\Cov(f_{1-\eps,Y_\eps^{(i)}})}_{op}\leq\frac{1}{\eps}.
\]
Then, by Theorem 15.10 in \cite{BLM13} and \cref{lem:kl},
\ba{
\left\|\sqrt{\sum_{i=1}^n\norm{\Cov(f_{1-\eps,Y_\eps^{(i)}})}_{op}^2}\right\|_p
\leq C\sqrt{n+\frac{p}{\eps}}.
}
Also, by Lemma 2.1 in \cite{FaKo22} and \cref{lem:kl}, 
\ba{
\left\|\left(\frac{1}{n}\sum_{i=1}^n\Cov(f_{1-\eps,Y_\eps^{(i)}})-\Sigma_\eps\right)u\right\|_p
\leq \frac{C|u|}{n}(\sqrt{pn}+p).
}
Moreover, $\|\Sigma_\eps^{-1/2}\|_{op}\leq C$ since $\Sigma_\eps\succeq\frac{2}{9}I_d$. Consequently, 
\ba{
\norm{u\cdot(W-Z)}_p
&\leq \frac{C|u|}{\sqrt{\varpi_d\log (2d)}}\frac{p^{5/2}}{n}
+C|u|\left(\sqrt{\varpi_d}\frac{p}{\sqrt n}+\varpi_d\log (2d)\frac{p^{3/2}}{n}\right).
}
This gives \eqref{lc:prowass}.
\medskip

%
%
%
%



\noindent
\textbf{Step 2.} Next we consider the general case. 
Let $\Sigma=U^\top\Lambda U$ be a a spectral decomposition of $\Sigma$, where $U$ is a $d\times d$ orthogonal matrix and $\Lambda$ is a $d\times d$ diagonal matrix. Without loss of generality, we may assume that the first $r$ diagonal entries $\lambda_1,\dots,\lambda_r$ of $\Lambda$ are positive and others are zero. 

Let $X\sim\mu$. Since $\Cov(UX)=\Lambda$, $(UX)_j=0$ a.s.~for $j=r+1,\dots,d$. Let $\nu$ be the law of the random vector $((UX)_1/\sqrt{\lambda_1},\dots,(UX)_r/\sqrt{\lambda_r})^\top$ in $\mathbb R^r$. By construction, $\nu$ is isotropic. Also, $\nu$ is log-concave by Proposition 3.1 in \cite{SaWe14}. Thus, we can construct random vectors $W_0$ and $Z_0$ in $\mathbb R^r$ such that $W_0\overset{d}{=}n^{-1/2}\sum_{i=1}^nY_i$ with $Y_i\overset{i.i.d.}{\sim}\nu$, $Z_0\sim N(0,I_r)$ and
\be{
\|v\cdot(W_0-Z_0)\|_p\leq C|v|\left(\sqrt{\varpi_r}\frac{p}{\sqrt n}
+\varpi_r\log (2r)\frac{p^{3/2}}{n}
+\frac{1}{\sqrt{\varpi_r\log (2r)}}\frac{p^{5/2}}{n}\right)
}
for all $v\in\mathbb R^r$ and $p\geq1$. 
Now, define an $r\times d$ matrix as $S=(\diag(\sqrt{\lambda_1},\dots,\sqrt{\lambda_r})~O_{r,d-r})$, where $O_{r,d-r}$ denotes the $r\times(d-r)$ zero matrix. 
Then, set $\tilde Y_i=S^\top Y_i$ for $i=1,\dots,n$. By construction, $\tilde Y_i$ has the same law as $UX$. Hence, $X_i:=U^\top \tilde Y_i\sim\mu$. Let $W=n^{-1/2}\sum_{i=1}^nX_i$. Also, set $Z=U^\top S^\top Z_0$. We have $Z\sim N(0,\Sigma)$ by construction. Finally, for any $u\in\mathbb R^d$ and $p\geq1$,
\ba{
\norm{u\cdot(W-Z)}_p
&=\norm{SUu\cdot(W_0-Z_0)}_p\\
&\leq C|SUu|\left(\sqrt{\varpi_r}\frac{p}{\sqrt n}
+\varpi_r\log (2r)\frac{p^{3/2}}{n}
+\frac{1}{\sqrt{\varpi_r\log (2r)}}\frac{p^{5/2}}{n}\right).
}
Since $|SUu|^2=u^\top U^\top S^\top SUu=u^\top U^\top\Lambda Uu=u^\top\Sigma u=|\sqrt{\Sigma}u|^2$, we complete the proof. 
\end{proof}

\section{Proof of \cref{md-max}}\label{sec:md-max}

We write $\phi$ and $\Phi$ for the density and distribution function of $N(0,1)$, respectively. 
The proof relies on the following lemma.
\begin{lemma}\label{lem:gmax-tail}
Let $Z$ be as in \cref{md-max}. Then
\[
P\left(x-\eps<\max_{1\leq j\leq d}Z_j\leq x\right)\leq\frac{\eps}{\ul\sigma}(1+x/\ul\sigma)\exp(\eps x/\ul\sigma^2)P\left(\max_{1\leq j\leq d}Z_j>x\right)
\]
for any $x\geq0$ and $\eps>0$.
\end{lemma}

\begin{proof}
The proof is a modification of that of Theorem 3 in \cite{CCK15}. Without loss of generality, we may assume that the correlation coefficient between $Z_j$ and $Z_k$ is less than 1 whenever $j\neq k$. 
Set $\tilde Z_j:=(Z_j-x)/\sigma_j+x/\ul\sigma$ for $j=1,\dots,d$. Then
\ba{
P\left(x-\eps<\max_{1\leq j\leq d}Z_j\leq x\right)
&=P\left(\left\{\bigcup_{j=1}^d\{Z_j>x-\eps\}\right\}\cap\left\{\bigcap_{j=1}^d\{Z_j\leq x\}\right\}\right)\\
&=P\left(\left\{\bigcup_{j=1}^d\{\tilde Z_j>x/\ul\sigma-\eps/\sigma_j\}\right\}\cap\left\{\bigcap_{j=1}^d\{\tilde Z_j\leq x/\ul\sigma\}\right\}\right)\\
&\leq P\left(\frac{x-\eps}{\ul\sigma}<\max_{1\leq j\leq d}\tilde Z_j\leq \frac{x}{\ul\sigma}\right).
}
Since $\tilde Z_j\sim N(\E\tilde Z_j,1)$ and $\E\tilde Z_j=-x/\sigma_j+x/\ul\sigma\geq0$ for all $j$, by Lemma 5 in \cite{CCK15}, $\max_{1\leq j\leq d}\tilde Z_j$ has density of the form $f(z)=\phi(z)G(z)$, where $G$ is non-decreasing by Lemma 6 in \cite{CCK15}. Then, for any $z\in\mathbb R$,
\ba{
\int_z^\infty\phi(u)duG(z)
&\leq\int_z^\infty\phi(u)G(u)du
=P\left(\max_{1\leq j\leq d}\tilde Z_j>z\right).
}
Hence
\ba{
\sup_{z\in[(x-\eps)/\ul\sigma,x/\ul\sigma]}f(z)&\leq\phi((x-\eps)_+/\ul\sigma)G(x/\ul\sigma)\leq\frac{\phi((x-\eps)_+/\ul\sigma)}{1-\Phi(x/\ul\sigma)}P\left(\max_{1\leq j\leq d}\tilde Z_j>x/\ul\sigma\right).
}
By Birnbaum's inequality, we have for all $z\geq0$
\ben{\label{024}
\frac{\phi(z)}{1-\Phi(z)}\leq\frac{2}{\sqrt{4+z^2}-z}=\frac{\sqrt{4+z^2}+z}{2}\leq1+z.
}
Hence
\ba{
\frac{\phi((x-\eps)_+/\ul\sigma)}{1-\Phi(x/\ul\sigma)}
\leq(1+x/\ul\sigma)\frac{\phi((x-\eps)_+/\ul\sigma)}{\phi(x/\ul\sigma)}
\leq(1+x/\ul\sigma)\exp(\eps x/\ul\sigma^2).
}
Also,
\ba{
P\left(\max_{1\leq j\leq d}\tilde Z_j>x/\ul\sigma\right)
=P\left(\bigcup_{j=1}^d\{\tilde Z_j>x/\ul\sigma\}\right)
=P\left(\max_{1\leq j\leq d}Z_j>x\right).
}
Hence we obtain the desired result. 
\end{proof}

\begin{proof}[Proof of \cref{md-max}]
In this proof, we use $C$ to denote positive constants, which depend only on $\alpha,A$ and $B$ and may be different in different expressions. 
First we prove the claim when $\Delta/\ul\sigma<1/e$. 
Set 
\be{
p=\log d+\log(\ul\sigma/\Delta)+\frac{x^2}{\ol\sigma^2}, \quad \eps=Ap^{\alpha} \Delta e.
} 
Because $\log d+|\log (\Delta/\ul\sigma)|\leq p_0/2$ and $x\leq \ol\sigma\sqrt{p_0/2}$, we have $p\leq p_0$. 

We have
\ba{
P\left(\max_jW_j>x\right)
&\leq P\left(\max_jZ_j>x-\eps\right)+P\left(\left|\max_jW_j-\max_jZ_j\right|>\eps\right)\\
&\leq P\left(\max_jZ_j>x\right)+P\left(x-\eps<\max_jZ_j\leq x\right)+P\left(\left|\max_jW_j-\max_jZ_j\right|>\eps\right).
}
By Markov's inequality and assumption,
\ba{
P\left(\left|\max_jW_j-\max_jZ_j\right|>\eps\right)
&\leq\eps^{-p}\E\left|\max_jW_j-\max_jZ_j\right|^p\\
&\leq\eps^{-p}\E\max_j|W_j-Z_j|^p\\
&\leq \eps^{-p}d(Ap^\alpha\Delta)^p
=de^{-p}=\frac{\Delta}{\ul\sigma} e^{-\frac{x^2}{\ol\sigma^2}}.
}
Also, by \cref{lem:gmax-tail},
\ba{
P\left(x-\eps<\max_jZ_j\leq x\right)
\leq \frac{\eps}{\ul\sigma}(1+x/\ul\sigma)\exp(\eps x/\ul\sigma^2)P\left(\max_jZ_j>x\right).
}
Hence
\ba{
P\left(\max_jW_j>x\right)
&\leq P\left(\max_jZ_j>x\right)+\frac{\eps}{\ul\sigma}(1+x/\ul\sigma)\exp(\eps x/\ul\sigma^2)P\left(\max_{j}Z_j>x\right)
+\frac{\Delta}{\ul\sigma} e^{-\frac{x^2}{\ol\sigma^2}}.
}
Similarly, we deduce
\ba{
&P\left(\max_jZ_j>x\right)\\
&=P\left(\max_jZ_j>x+\eps\right)+P\left(x<\max_jZ_j\leq x+\eps\right)\\
&\leq P\left(\max_jW_j>x\right)+P\left(\left|\max_jW_j-\max_jZ_j\right|>\eps\right)+P\left(x<\max_jZ_j\leq x+\eps\right)\\
&\leq P\left(\max_jW_j>x\right)+\frac{\eps}{\ul\sigma}(1+(x+\eps)/\ul\sigma)\exp(\eps (x+\eps)/\ul\sigma^2)P\left(\max_{j}Z_j>x+\eps\right)
+\frac{\Delta}{\ul\sigma} e^{-\frac{x^2}{\ol\sigma^2}}.
}
Consequently, we obtain
\ba{
\left|P\left(\max_jW_j>x\right)-P\left(\max_jZ_j>x\right)\right|
\leq \frac{\eps}{\ul\sigma}(1+(x+\eps)/\ul\sigma)\exp(\eps (x+\eps)/\ul\sigma^2)P\left(\max_{j}Z_j>x\right)
+\frac{\Delta}{\ul\sigma} e^{-\frac{x^2}{\ol\sigma^2}}.
}
Since $x\leq \ul\sigma(\ul\sigma/\Delta)^{1/(2\alpha+1)}$ and $\log d\leq C(\ul\sigma/\Delta)^{2/(2\alpha+1)}$, we have
\besn{\label{eps-est-max}
\eps
&\leq C\Delta((\log d)^{\alpha}+\{\log(\ul\sigma/\Delta)\}^{\alpha}+x^{2\alpha}/\ol\sigma^{2\alpha})\\
&\leq C\Delta(\{\log(\ul\sigma/\Delta)\}^{\alpha}+(\ul\sigma/\Delta)^{2\alpha/(2\alpha+1)})
\leq C\ul\sigma(\Delta/\ul\sigma)^{1/(2\alpha+1)}.
}
Hence $\exp(\eps (x+\eps)/\ul\sigma^2)\leq C$ and $\eps/\ul\sigma\leq C$. 
Also, letting $J$ be an element of $\{1,\dots,d\}$ satisfying $\ol\sigma^2=\E Z_J^2$, we have
\ben{\label{lb-gtail}
P\left(\max_jZ_j>x\right)\geq P\left(Z_J>x\right)=1-\Phi(x/\ol\sigma)\geq\frac{\phi(x/\ol\sigma)}{1+x/\ol\sigma}.
}
Hence
\ba{
&\left|P\left(\max_jW_j>x\right)-P\left(\max_jZ_j>x\right)\right|\\
&\leq C\frac{\eps}{\ul\sigma}(1+x/\ul\sigma)P\left(\max_jZ_j>x\right)
+\sqrt{2\pi}e^{-\frac{x^2}{2\ol\sigma^2}}(1+x/\ol\sigma)\frac{\Delta}{\ul\sigma} P\left(\max_jZ_j>x\right)\\
&\leq \frac{C}{\ul\sigma}\{\eps(1+x/\ul\sigma)+\Delta\}P\left(\max_jZ_j>x\right).
}
This completes the proof of \eqref{eq:mdp-max}. 

It remains to prove \eqref{eq:mdp-max} when $\Delta/\ul\sigma>1/e$. In this case, we have $x/\ul\sigma\leq C$. Thus, by \eqref{lb-gtail},
\[
\frac{1}{P(\max_jZ_j>x)}\leq C.
\]
So \eqref{eq:mdp-max} holds because $\Delta/\ul\sigma>1/e$.
\end{proof}

\begin{appendix}
\section{Uniqueness in law for \eqref{eq:follmer2}}\label{sec:appendix}

\begin{theorem}\label{thm:unique}
Uniqueness in law holds for the SDE \eqref{eq:follmer2} when $f$ is the smooth, positive density with respect to $N(0,I_d)$ of a log-concave probability measure $\mu$. 
\end{theorem}

\begin{proof}
Let $Y^{(i)}=(Y^{(i)}_t)_{t\in[0,1]}$, $i=1,2$, be weak solutions to \eqref{eq:follmer2}. Since both $Y^{(1)}$ and $Y^{(2)}$ are continuous, it suffices to show that $(Y^{(1)}_t)_{t\in[0,t_1]}$ and $(Y^{(2)}_t)_{t\in[0,t_1]}$ have the same law for any fixed $\frac{1}{2}<t_1<1$. By Proposition 3.10 in \cite[Chapter 5]{KaSh98}, this follows once we show that
\[
\int_0^{t_1}|\nabla\log P_{1-t}f(Y^{(i)}_t)|^2dt<\infty\quad\text{a.s.}\quad\text{for }i=1,2.
\]
Set $\rho=f\phi_{I_d}$ so that $\rho$ is the Lebesgue density of $\mu$. By Theorem 5.1 in \cite{SaWe14}, there exist constants $a,b>0$ such that $\rho(x)\leq e^{-a|x|+b}$ for all $x\in\mathbb R^d$. 
Also, observe that 
\[
\nabla\log P_{1-t}f(y)=\frac{\nabla P_{1-t}f(y)}{P_{1-t}f(y)}=\frac{\int_{\mathbb R^d}\frac{x-y}{1-t}\rho(x)e^{\frac{2x\cdot y-t|x|^2}{2(1-t)}}dx}{\int_{\mathbb R^d}\rho(x)e^{\frac{2x\cdot y-t|x|^2}{2(1-t)}}dx}\quad\text{for }t\in(0,1),y\in\mathbb R^d.
\]
By this expression, for any $\tau\in(0,t_1)$ and $K>0$, we have 
\ben{\label{score-large-t}
\sup_{t\in[\tau,t_1],|y|\leq K}|\nabla\log P_{1-t}f(y)|<\infty.
}
Moreover, if $\tau\leq1/2$,
\ben{\label{score-small-t}
\sup_{t\in[0,\tau],|y|\leq a/4}|\nabla\log P_{1-t}f(y)|
\leq\frac{\int_{\mathbb R^d}\frac{|x|+|a|/4}{1-\tau}e^{-\frac{a}{2}|x|+b}dx}{\int_{\mathbb R^d}\rho(x)e^{\frac{-a|x|/2-\tau|x|^2}{2(1-\tau)}}dx}<\infty.
}
Now, since $Y^{(i)}_t\to Y^{(i)}_0=0$ a.s. as $t\downarrow0$, there is a random variable $t_0\in(0,1/2)$ such that $\sup_{0\leq t\leq t_0}|Y^{(i)}_t|\leq a/4$. Then, noting $\sup_{t\in[0,1]}|Y_t^{(i)}|<\infty$, we have $\sup_{t\in[t_0,t_1]}|\nabla\log P_{1-t}f(Y^{(i)}_t)|<\infty$ a.s.~by \eqref{score-large-t} and $\sup_{t\in[0,t_0]}|\nabla\log P_{1-t}f(Y^{(i)}_t)|<\infty$ a.s.~by \eqref{score-small-t}. Consequently, we obtain
\[
\int_0^{t_1}|\nabla\log P_{1-t}f(Y^{(i)}_t)|^2dt\leq\sup_{t\in[0,t_1]}|\nabla\log P_{1-t}f(Y^{(i)}_t)|^2<\infty\quad\text{a.s.}
\]
This completes the proof. 
\end{proof}
\end{appendix}

\section*{Acknowledgements}

The authors would like to thank two anonymous referees for their careful reading of the manuscript, many constructive comments, and indication of errors that were contained in the original version of the manuscript. 
Fang X. was partially supported by Hong Kong RGC GRF 14302418, 14305821, a CUHK direct grant and a CUHK start-up grant. Koike Y. was partly supported by JST CREST Grant Number JPMJCR2115 and JSPS KAKENHI Grant Numbers JP19K13668, JP22H00834, JP22H01139.


\end{document}